\documentclass[10pt]{article}
\usepackage[a4paper]{geometry}

\usepackage{amsmath,amsthm,amssymb,bm,xspace,enumerate,color}
\numberwithin{equation}{section}
\usepackage{graphicx,verbatim,mathtools}
\usepackage[only,llbracket,rrbracket,llparenthesis,rrparenthesis]{stmaryrd}
\usepackage{mathrsfs}
\usepackage{tikz,pgfplots}
\pgfplotsset{compat=newest}
\usepackage[textsize=tiny,color=blue!10!white]{todonotes}
\usepackage[breaklinks,bookmarks=false]{hyperref}
\hypersetup{colorlinks, linkcolor=blue, citecolor=blue,
urlcolor=blue, plainpages=false, pdfwindowui=false,
pdfstartview={FitH}}

\newtheorem{theorem}{Theorem}[section]{\bfseries}{\it}
{\bfseries}{\it}
\newtheorem{lemma}[theorem]{Lemma}{\bfseries}{\it}
\newtheorem{corollary}[theorem]{Corollary}{\bfseries}{\it}
{\bfseries}{\it}
{\bfseries}{\it}
\theoremstyle{definition}
\newtheorem{remark}{Remark}[section]{\bfseries}{\rmfamily}

\renewcommand{\dim}{d}
\newcommand{\p}{\partial}
\newcommand{\abs}[1]{\lvert#1\rvert} 
\newcommand{\tends}{\rightarrow}
\newcommand{\norm}[1]{\lVert#1\rVert}
\newcommand{\N}{\mathbb{N}}

\newcommand{\cF}{\mathcal{F}}
\newcommand{\cT}{\mathcal{T}}

\newcommand{\eps}{\epsilon}

\newcommand{\Om}{\Omega}
\newcommand{\om}{\omega}
\newcommand{\sA}{\mathscr{A}}
\newcommand{\sB}{\mathscr{B}}

\newcommand{\avg}[1]{\left\{#1\right\}}
\newcommand{\jump}[1]{\llbracket #1 \rrbracket}
\DeclareMathOperator{\Tr}{Tr}
\DeclareMathOperator{\diam}{diam}
\DeclareMathOperator{\Div}{div}

\newcommand{\bn}{\bm{n}}

\newcommand{\Npw}{\nabla}
\newcommand{\Dpw}{\nabla^2}

\newcommand{\Pp}{\mathbb{P}_p}

\newcommand{\normt}[1]{\norm{#1}_{\cT}}

\newcommand{\hT}{h_{\cT}}

\newcommand{\Dp}{\Npw}

\newcommand{\bphi}{\bm{\phi}}

\newcommand{\Jpen}{J_{\cT}}
\newcommand{\Fg}{F_{\gamma}}
\newcommand{\ab}{{\alpha\beta}}
\newcommand{\R}{\mathbb{R}}
\renewcommand{\dim}{d}
\newcommand{\LL}{L_\lambda}
\newcommand{\shapereg}{\vartheta_\cT}
\newcommand{\absJ}[1]{\abs{#1}_{J,\cT}}
\newcommand{\Vt}{V_\cT}
\newcommand{\eT}{\eta_{\cT}}
\newcommand{\Dmut}{\Delta_{\cT}}
\newcommand{\St}{S_{\cT}}
\newcommand{\wt}{w_{\cT}}
\newcommand{\vt}{v_{\cT}}
\newcommand{\ut}{u_{\cT}}
\newcommand{\LLt}{L_{\lambda,\cT}}
\newcommand{\smin}{\sigma_{\mathrm{min}}}
\newcommand{\rhomin}{\rho_{\mathrm{min}}}
\newcommand{\bVt}{\bm{V}_{\cT}}
\newcommand{\bH}{\bm{H}_T^1(\Om)}
\newcommand{\bE}{\bm{E}_{\cT}}
\newcommand{\bCt}{\bm{C}_{\cT}}
\newcommand{\bwt}{\bm{w}_{\cT}}
\newcommand{\bvt}{\bm{v}_{\cT}}
\newcommand{\zt}{z_{\cT}}
\newcommand{\normbt}[1]{\norm{#1}_{\bVt}}
\newcommand{\absbJ}[1]{\abs{#1}_{\bm{J},\cT}}
\newcommand{\absL}[1]{\abs{#1}_{\lambda,\cT}}
\newcommand{\At}{A_{\cT}}
\newcommand{\Cdiam}{C_{\dim,\diam\Om}}
\newcommand{\Cmon}{C_{\mathrm{mon}}}
\newcommand{\CPF}{C_{\mathrm{PF}}}
\newcommand{\Hd}{H^2(\Om)\cap H^1_0(\Om)}

\title{Unified analysis of discontinuous Galerkin and $C^0$-interior penalty finite element methods for Hamilton--Jacobi--Bellman and Isaacs equations\footnotemark[1]}
\author{Ellya L.~Kawecki\footnotemark[2] and Iain Smears\footnotemark[2]}

\begin{document}

\maketitle

\renewcommand{\thefootnote}{\fnsymbol{footnote}}

\footnotetext[1]{This work was supported by the Engineering and Physical Sciences Research Council (EPSRC) Doctoral Prize Fellowship grant EP/R513143/1.}

\footnotetext[2]{Department of Mathematics, University College London, Gower
Street, WC1E 6BT London, United Kingdom (\texttt{e.kawecki@ucl.ac.uk}, \texttt{i.smears@ucl.ac.uk}).}

\begin{abstract}
We provide a unified analysis of \emph{a posteriori} and \emph{a priori} error bounds for a broad class of discontinuous Galerkin and $C^0$-IP finite element approximations of fully nonlinear second-order elliptic Hamilton--Jacobi--Bellman and Isaacs equations with Cordes coefficients. 
We prove the existence and uniqueness of strong solutions in $H^2$ of Isaacs equations with Cordes coefficients posed on bounded convex domains. 
We then show the reliability and efficiency of computable residual-based error estimators for piecewise polynomial approximations on simplicial meshes in two and three space dimensions. 
We introduce an abstract framework for the \emph{a priori} error analysis of a broad family of numerical methods and prove the quasi-optimality of discrete approximations under three key conditions of Lipschitz continuity, discrete consistency and strong monotonicity of the numerical method.
Under these conditions, we also prove convergence of the numerical approximations in the small-mesh limit for minimal regularity solutions.
We then show that the framework applies to a range of existing numerical methods from the literature, as well as some original variants. A key ingredient of our results is an original analysis of the stabilization terms.
As a corollary, we also obtain a generalization of the discrete Miranda--Talenti inequality to piecewise polynomial vector fields.
\end{abstract}

\section{Introduction}

We consider fully nonlinear second-order elliptic Isaacs equations with a homogeneous Dirichlet boundary condition of the form
\begin{equation}\label{eq:isaacs_pde}
\begin{aligned}
F[u]\coloneqq \inf_{\alpha\in\sA}\sup_{\beta\in\sB}\left[  L^{\ab} u-f^{\ab} \right] & = 0 &&\text{in }\Om,\\
u & = 0 &&\text{on }\partial\Om,
\end{aligned}
\end{equation}
where $\Om$ is a bounded convex polytopal open set in $\R^\dim$, $\dim\in\{2,3\}$ and where the second-order elliptic operators $L^{\ab}$ are defined in~\eqref{eq:Lab_def} below.
It is also possible to consider the case where the order of the infimum and supremum in~\eqref{eq:isaacs_pde} are reversed.
Isaacs equations of the form~\eqref{eq:isaacs_pde} arise in applications of two-player games of stochastic optimal control, and they can be seen as a generalization of Hamilton--Jacobi--Bellman (HJB) equations~\cite{FlemingSoner06}.
Isaacs and HJB equations and related stochastic control problems arise in many applications from engineering, energy, finance and computer science. Many other important nonlinear partial differential equations (PDE) can be reformulated as HJB or Isaacs equations, including the Monge--Amp\`ere equation which, along with its convexity constraint, can be reformulated as a fully nonlinear HJB equation as shown in~\cite{FengJensen17,Krylov87}; see also~\cite{Kawecki2018a} for some further results.
The equation in~\eqref{eq:isaacs_pde} is fully nonlinear in the sense that all partial derivatives are contained in the nonlinearity, which prohibits approaches based on weak solutions that are standard for divergence form problems.

The design and analysis of stable and accurate numerical methods for the approximation of the solution of fully nonlinear PDE such as~\eqref{eq:isaacs_pde} remains generally very challenging. 
One approach consists of designing methods that satisfy a discrete maximum principle, which can be shown to converge to the viscosity solution in the maximum norm under appropriate conditions of consistency, stability and the availability of a comparison principle for viscosity sub- and supersolutions~\cite{CrandallIshiiLions92,Souganidis91}. 
See~\cite{Caffarelli1995} and the references therein for the regularity theory of viscosity solutions.
Efforts in this direction have focused primarily on finite difference methods~\cite{DebrabantJakobsen13,FengJensen17,KuoTrudinger1990,KushnerDupuis01}, although there has been recent interest also in finite element methods (FEM) satisfying a maximum principle~\cite{Jensen17,JensenS13}, which additionally show stability and convergence of the derivatives in $L^2$.
See also~\cite{NochettoZhang18,SalgadoZhang19} for methods based on integral-operator approximations.
Methods based on discrete maximum principles have the advantage of being able to handle problems with possibly degenerate second-order terms and correspondingly low-regularity solutions. 
However, it is well-known that enforcing a discrete maximum principle is restrictive in practice, typically requiring highly structured grids or meshes and wide stencil approximations of the differential operators, and it also leads to limitations on the order of convergence~\cite{BonnansZidani03,CrandallLions96,Kocan95,MotzkinWasow53}.

There is therefore considerable interest in the analysis of methods that do not require a discrete maximum principle~\cite{Brenner2011,LakkisPryer11,LakkisPryer13,FengGlowinskiNeilan13,NeilanSalgadoZhang2017}, although a long-standing difficulty has been to design provably stable methods for a sufficiently broad range of problems.
This challenge was resolved in~\cite{SS13,SS14,SS16} in the context of nondivergence form elliptic equations and fully nonlinear HJB equations on convex domains that satisfy the Cordes condition. 
The Cordes condition is an algebraic assumption on the coefficients of the linear operators inside the nonlinear terms, which is thus naturally preserved under linearizations of the original fully nonlinear operator and also under discretization.
The motivation for the Cordes condition stems from the analysis of linear nondivergence form elliptic equations with discontinuous coefficients, which arise as linearizations of fully nonlinear HJB equations under policy iteration.
In particular, it is well-known that for linear nondivergence form elliptic equations in three space dimensions and above, the discontinuities in the diffusion coefficients generally lead to ill-posedness, even in the uniformly elliptic case with smooth data on a smooth domain, and for both strong and viscosity solutions with measurable ingredients~\cite{Safonov1999,MaugeriPalagachevSoftova00}.
Further assumptions on the coefficients (other than continuity) are therefore generally necessary to recover well-posedness.
For instance, there are available results on the well-posedness of strong solutions when the coefficients are of \emph{vanishing mean-oscillation}~\cite{ChiarenzaFrascaLongo93,MaugeriPalagachevSoftova00}; however in practice, the discontinuous coefficients obtained in the linearized problems mentioned above typically feature jump discontinuities and are not of vanishing mean-oscillation.
The case of general $L^\infty$ coefficients thus falls outside the scope of the Calder\'on--Zygmund and Schauder theories~\cite{GilbargTrudinger2001}.
In particular, it can be shown that well-posedness is recovered for strong solutions on convex domains under the Cordes condition~\cite{Cordes1956}, see also~\cite{MaugeriPalagachevSoftova00} for a comprehensive discussion. 
In two space dimensions however, uniform ellipticity implies the Cordes condition.
It was then shown in~\cite{SS13,SS14} that the Cordes condition also implies existence and uniqueness of strong solutions in $H^2$ for fully nonlinear second-order elliptic HJB equations on convex domains, where an $hp$-version discontinuous Galerkin (DG) finite element method was proposed with proven stability and with optimal convergence rates with respect to the mesh-size, and half-order suboptimal rates with respect to the polynomial degree, in $H^2$-type norms. 
It was also shown in \cite{SS14} that policy iteration, understood as a semismooth Newton method, has local superlinear convergence.
These results were extended to parabolic problems in~\cite{SS16}. 
There has since been significant recent activity centred on this approach, including preconditioners~\cite{S18}, adaptive $H^2$-conforming and mixed methods in~\cite{Gallistl17,Gallistl19}, extensions to curved domains~\cite{Kawecki19b}, boundary conditions involving oblique derivatives~\cite{Gallistl19b,Kawecki19}, and $C^0$ interior penalty (IP) methods~\cite{Bleschmidt19,Kawecki19c,NeilanWu19}. 

In this work, we present a unified \emph{a priori} and \emph{a posteriori} analysis of DG and $C^0$-IP methods for Isaacs equations~\eqref{eq:isaacs_pde} with Cordes coefficients.
First, we extend the well-posedness result of~\cite{SS14} for fully nonlinear HJB equations to the setting of Isaacs equations, showing existence and uniqueness of a strong solution of~\eqref{eq:isaacs_pde} in $H^2(\Om)\cap H^1_0(\Om)$.
This is the subject of Section~\ref{sec:pde}. Our second main contribution is a proof of \emph{reliability} and \emph{local efficiency} of residual-based error estimators in $H^2$-norms for piecewise polynomial approximations on simplicial meshes, which consist of unweighted volume residuals with appropriately penalized jumps of function gradients and jumps of function values.
This extends earlier results for $H^2$-conforming and $C^0$-IP methods from~\cite{
Bleschmidt19,Kawecki19c,Gallistl17}.
In fact, owing to the strong solution of the PDE, we show that the \emph{a posteriori} error analysis is determined primarily by the choice of approximation space and is otherwise independent of the numerical method, so that our a posteriori error bounds applies to \emph{any} piecewise polynomial function over the mesh.
This situation thus differs significantly from residual-based error estimates for divergence form elliptic problems, where the reliability bound is typically only satisfied under a suitable form of Galerkin orthogonality for the numerical solution~\cite{Verfurth13,CarstensenGudiJensen09}.
The above observation implies that our \emph{a posteriori} error analysis applies to any numerical method employing piecewise polynomial approximations on simplicial meshes.

Our further main contributions concern the \emph{a priori} error analysis of DG and $C^0$-IP methods for Isaacs equations.
We provide a framework for proving quasi-optimality, also called near-best approximation, of the error attained by the numerical solution under only the minimum guaranteed regularity of the solution in $H^2(\Om)\cap H^1_0(\Om)$. 
The key requirements on the numerical method of the framework are Lipschitz continuity, strong monotonicity and an appropriate notion of consistency.
Therefore, this generalizes C\'ea's Lemma to the problem at hand, which, interestingly for nonconforming methods, holds here without additional terms related to data oscillation~\cite{Gudi2010}.
We then prove convergence of the numerical approximations in the small-mesh limit for sequences of shape-regular meshes, without any additional regularity assumptions.
We then show how our framework applies to a broad family of DG and $C^0$-IP methods which include as special cases the methods of~\cite{SS13,SS14} (restricted here to simplicial meshes and fixed polynomial degrees), the method of~\cite{NeilanWu19}, as well as some original variants that are of further interest in the context of adaptive methods~\cite{KaweckiSmears20adapt}.
Thus, up to the constants involved, all of these methods are quasi-optimal and converge in the minimal regularity setting.
We note from the onset that we consider here a homogeneous boundary condition for simplicity, and that nonhomogeneous boundary data can be also be handled with minor adjustments, see e.g.\ \cite[Section~6.2]{SS13} for some further discussion.

These results are original even in the setting of HJB equations, and our current approach to the \emph{a priori} error analysis differs significantly from the earlier approach of~\cite{SS13,SS14}. 
Indeed, in \cite{SS13,SS14} the analysis employs a notion of consistency that involves the insertion of the exact solution of the problem into the discrete forms, which leads to additional regularity assumptions on the exact solution in order to handle terms involving traces of second derivatives on mesh faces, see e.g.~\cite[Corollary~6]{SS13}.
In this work, we propose and show a different notion of consistency, that is determined entirely at the discrete level (thus called here~\emph{discrete consistency}) and thus does not involve additional assumptions on the exact solution.
The key to showing that the methods satisfy the discrete consistency condition is an original sharp analysis of the kernel of the stabilization terms that were first introduced in~\cite{SS13}, see in particular Theorem~\ref{thm:stab_bound} below.
Note that methods using the original stabilization terms of~\cite{SS13,SS14} remain competitive in practice owing to the fact that they lead to penalization parameters that are robust with respect to domain geometry, and they have further advantages in terms of flexibility, since they can accommodate extensions to $hp$-version, meshes with hanging nodes, non-simplicial elements, etc. 
We also show here that the discrete Miranda--Talenti inequality of~\cite{NeilanWu19} can be seen as a special case of a more general result for piecewise polynomial discontinuous vector fields.

This paper is organized as follows. First, we prove the well-posedness of~\eqref{eq:isaacs_pde} on convex domains under the Cordes condition in Section~\ref{sec:pde}. Then, after defining the notation in Section~\ref{sec:notation}, we present the general \emph{a posteriori} and \emph{a priori} error analysis in Section~\ref{sec:framework}. In section~\ref{sec:num_schemes} we present the family of numerical methods, and present our main results that verify the abstract assumptions of the framework. The proofs, including the analysis of the stabilization terms and discrete Miranda--Talenti inequalities, are then given in~Section~\ref{sec:proofs}.

\section{Analysis of well-posedness of the problem}\label{sec:pde}

Let $\Om\subset\mathbb \R^\dim$, $\dim\in\{2,3\}$, be a bounded convex polytopal open set.
The assumption $\dim\in\{2,3\}$ is primarily technical and is related to some $H^2$-enrichment operators that appear later in this work in Section~\ref{sec:abtract_apost}. We therefore note that the results of this section are not restricted to $\dim\in \{2,3\}$ and in fact hold for general dimensions.
Let $\sA,\sB$ be compact metric spaces, and let the $\R_{\mathrm{sym}}^{\dim\times\dim}$ matrix-valued function $a$, the $\R^\dim$ vector-valued function $b$, and the real-valued functions $c$ and $f$ be continuous on $\overline{\Om}\times\sA\times\sB$, where $\R_{\mathrm{sym}}^{\dim\times\dim}$ denotes the space of symmetric $\dim\times\dim$ matrices.
For each $(\alpha,\beta)\in\sA\times\sB$ we define $a^\ab\colon x \mapsto a(x,\alpha,\beta)$ for all $x\in\overline{\Om}$. 
The functions $b^{\ab}$, $c^{\ab}$ and $f^{\ab}$ are defined in a similar manner for each $(\alpha,\beta)\in\sA\times\sB$.
It is assumed that $c^\ab$ is nonnegative in $\Om$ for all $(\alpha,\beta)\in\sA\times \sB$, and that the diffusion coefficients $a^\ab$ are uniformly elliptic, uniformly over $\sA\times\sB$, i.e.\ there exist positive constants $\underline{\nu}$ and $\overline{\nu}$ such that
\begin{equation}\label{eq:ellipticity}
\begin{aligned}
\underline{\nu}|\xi|^2\le\xi^\top a^{\ab}(x)\xi\leq \overline{\nu} |\xi|^2 &&&\forall x\in\Om,\,\forall\xi\in\mathbb R^d,\forall(\alpha,\beta)\in\sA\times\sB,
\end{aligned}
\end{equation}
where $\abs{\xi}$ denotes the Euclidean norm of the vector $\xi\in\R^\dim$.

If the functions $b$ and $c$ both vanish identically on $\overline{\Om}\times\sA\times\sB$, i.e.\ $b\equiv 0$ and $c\equiv 0$, then we assume the Cordes condition: there exists a $\nu\in(0,1]$ such that
\begin{equation}\label{eq:Cordes2}
\begin{aligned}
\frac{\abs{a^{\ab}}^2}{\Tr(a^{\ab})^2}\le\frac{1}{d-1+\nu}&&&\text{in }\Om\quad\forall(\alpha,\beta)\in\sA\times\sB,
\end{aligned}
\end{equation}
where $\abs{a^{\ab}}$ denotes the Frobenius norm of the matrix $a^\ab$.
Otherwise, in the case of nonvanishing lower-order terms, i.e.\ $b\not\equiv 0$ or $c\not\equiv 0$, we assume that there exists a $\lambda>0$ and a $\nu\in(0,1]$ such that
 \begin{equation}\label{eq:Cordes}
\begin{aligned}
\frac{\abs{a^{\ab}}^2+\abs{b^{\ab}}^2/2\lambda+(c^{\ab}/\lambda)^2}{(\Tr(a^{\ab})+c^{\ab}/\lambda)^2}\le\frac{1}{d+\nu}&&&\text{in }\Om\quad\forall(\alpha,\beta)\in\sA\times\sB,
\end{aligned}
\end{equation}
where $\abs{b^{\ab}}$ denotes the Euclidean norm of $b^\ab$.
As explained in~\cite{SS14} the parameter $\lambda$ serves to make the Cordes condition invariant under isotropic affine mappings of the domain. If $b$ and $c$ vanish identically, we let $\lambda =0$.

\begin{remark}
It is well-known that if $\dim=2$, then the uniform ellipticity condition~\eqref{eq:ellipticity} implies the Cordes condition~\eqref{eq:Cordes2}, and that $\nu$ can be bounded from below in terms of $\underline{\nu}$ and $\overline{\nu}$ alone, see for instance~\cite[Example~2]{SS14}.
\end{remark}

For each $(\alpha,\beta)\in\sA\times\sB$, the bounded linear operator $L^{\ab}\colon H^2(\Om) \tends L^2(\Om)$ is defined by 
\begin{equation}\label{eq:Lab_def}
\begin{aligned}
L^{\ab} v \coloneqq a^{\ab}{:}\nabla^2 v + b^{\ab}{\cdot}\nabla v - c^{\ab} v &&&\forall v\in H^2(\Om),
\end{aligned}
\end{equation}
where $\nabla^2 v$ denotes the Hessian of $v$, and where $A{:}B\coloneqq\sum_{i,j}^{\dim}A_{ij}B_{ij}$ denotes the Frobenius inner-product of matrices. 
The compactness of $\overline{\Om}\times\sA\times\sB$ and the continuity of the coefficients $a$, $b$, $c$ and $f$ imply that the fully nonlinear differential operator 
\begin{equation}\label{eq:F_def}
\begin{aligned}
 F[v]\coloneqq \inf_{\alpha\in\sA}\sup_{\beta\in\sB}\left[ L^{\ab} v - f^{\ab} \right] &&& \forall v \in H^2(\Om),
\end{aligned}
\end{equation}
is well-defined as a mapping from $H^2(\Om)$ to $L^2(\Om)$.
In~\cite{SS14,SS16} it was shown that fully nonlinear HJB equations can be reformulated in terms of a renormalized nonlinear operator. We show here that this approach extends to Isaacs equations.
For each $(\alpha,\beta)\in\sA\times\sB$, we consider the renormalization function $\gamma^{\ab}\in C(\overline{\Om})$ defined by $\gamma^{\ab}\coloneqq \frac{\Tr a^{\ab} }{\abs{a^{\ab}}^2}$ if the coefficients $b$ and $c$ vanish identically, or otherwise by
\begin{equation}
\gamma^{\ab}\coloneqq \frac{\Tr a^{\ab}+c^{\ab}/\lambda }{\abs{a^{\ab}}^2+\abs{b^{\ab}}^2/2\lambda+\abs{c^{\ab}}^2/\lambda^2}.
\end{equation}
In all cases, note that the continuity of the coefficients, the uniform ellipticity condition~\eqref{eq:ellipticity} and the nonnegativity of $c^{\ab}$ imply that there exists a uniform positive upper and lower bounds $\gamma^*$ and $\gamma_*>0$ such that $\gamma^*\geq \gamma^\ab\geq \gamma_*$ in $\overline{\Om}$ for all $(\alpha,\beta)\in\sA\times\sB$.
Let the renormalized operator~$\Fg \colon H^2(\Om)\tends L^2(\Om)$ be defined by
\begin{equation}\label{eq:Fg_def}
\begin{aligned}
\Fg[u]\coloneqq \inf_{\alpha\in\sA}\sup_{\beta\in\sB}\left[\gamma^{\ab}\left(L^\ab v - f^\ab\right)\right] &&&\forall v \in H^2(\Om).
\end{aligned}
\end{equation}
The following Lemma shows that the equations $F[u]=0$ and $\Fg[u] =0$ have equivalent respective sets of sub- and supersolutions.
\begin{lemma}\label{lem:subsupersolutions}
A function $v\in H^2(\Om)$ satisfies $F[v]\leq 0$ pointwise a.e.\ in $\Omega$ if and only if $\Fg[v]\leq 0$ pointwise a.e.\ in~$\Omega$. Furthermore, a function $v\in H^2(\Om)$ satisfies $F[v]\geq 0$ pointwise a.e.\ in $\Omega$ if and only if $\Fg[v]\geq 0$ pointwise a.e.\ in~$\Omega$.
\end{lemma}
\begin{proof}
The proof is a straightforward extension of the arguments in the proof of \cite[Theorem~3]{SS14}, and is primarily a consequence of the strict positivity of the renormalization function $\gamma^{\ab}$.
For each $\alpha\in\sA$, define the operators $G^{\alpha}[v]\coloneqq \sup_{\beta\in\sB}\left[L^\ab v- f^\ab \right]$ and $G_{\gamma}^{\alpha}[v]\coloneqq \sup_{\beta\in\sB}\left[\gamma^{\ab}(L^\ab v- f^\ab)\right]$ for each $v\in H^2(\Om)$.
We start by showing the equivalence of the sets of supersolutions.
Suppose that $v\in H^2(\Om)$; then $F[v]\geq 0$ a.e.\ in $\Om$ if and only if $G^\alpha[v]\geq 0$ a.e.\ in $\Om$ for every $\alpha\in \sA$.
Then, for any $\alpha\in\sA$, owing to compactness of $\sB$ and the continuity of the data, at almost every point $x\in\Omega$, the supremum in $G^{\alpha}[v](x)$ is attained by some $\beta^*\in\sB$, which gives $(L^{\alpha\beta^*}v-f^{\alpha\beta^*})(x)\geq 0$, which implies $G_\gamma^{\alpha}[v](x)\geq \gamma^{\alpha\beta^*}(L^{\alpha\beta^*}v-f^{\alpha\beta^*})(x)\geq 0$ using (strict) positivity of $\gamma^{\alpha\beta}$.
Considering also the converse situation, we then deduce that $G^\alpha[v]\geq 0$ a.e.\ in $\Omega$ is equivalent to $G_\gamma^{\alpha}[v]\geq 0$ a.e.\ in $\Omega$ for any $\alpha\in\sA$. Since $\alpha$ is arbitrary, we find that $F[v]\geq 0$ a.e.\ in $\Omega$ if and only if $\Fg[v]\geq 0$ a.e.\ in $\Omega$. We now consider the sets of subsolutions.
A function $v\in H^2(\Om)$ satisfies $F[v]\leq 0$ a.e. in $\Om$ if and only if, for a.e.\ $x\in\Om$, there exists an $\alpha_*\in \sA$ such that $G^{\alpha_*}[v](x) \leq 0$, which is equivalent to $(L^{\alpha_* \beta}v-f^{\alpha_*\beta})(x) \leq 0 $ for all $\beta\in\sB$, which is equivalent to $\gamma^{\alpha_*\beta}(x)(L^{\alpha_* \beta}v-f^{\alpha_*\beta})(x) \leq 0$ for all $\beta\in\sB$ by strict positivity of $\gamma^{\alpha_*\beta}$, which is finally equivalent to $G^{\alpha_*}_{\gamma}[v](x)\leq 0$. This shows that $G^{\alpha_*}[v] \leq 0$ a.e.\ in $\Om$ if and only if $G^{\alpha_*}_{\gamma}[v]\leq 0 $ a.e.\ in $\Om$, and thus the equivalence of $F[v]\leq 0$ a.e.\ in $\Om$ if and only if $\Fg[v]\leq 0$ a.e.\ in $\Om$, thereby completing the proof.
\end{proof}

A particular consequence of Lemma~\ref{lem:subsupersolutions} is that a solution of $F[u]=0$ is equivalently a solution of $\Fg[u]=0$.

\begin{remark}[Equivalence of problems in the sense of viscosity solutions]
The proof of Lemma~\ref{lem:subsupersolutions} involves only manipulations of pointwise values of the nonlinear operators $F$ and $\Fg$. Therefore, the $H^2$-regularity assumption on the sets of sub- and supersolutions in Lemma~\ref{lem:subsupersolutions} is not essential. In particular, recalling the notions of viscosity sub- and supersolutions~\cite{CrandallIshiiLions92}, it is easy to see that the argument above imply the equivalence of the sets of viscosity sub- and supersolutions (and hence also viscosity solutions) for the equations $F[u]=0$ and $\Fg[u]=0$.
\end{remark}

Let the differential operator $\LL\colon H^2(\Om)\tends L^2(\Om)$ be defined by
\begin{equation}
\begin{aligned}
\LL v \coloneqq \Delta v - \lambda v &&& \forall v \in H^2(\Om).
\end{aligned}
\end{equation}
We now show some bounds for the operator $\Fg$, including a Lipschitz continuity bound with a constant independent of the data. Since the properties are pointwise, we extend the definition of the operators $\Fg$ and $L_\lambda$ from the space $H^2(\Om)$ to $H^2(\om)$ for arbitrary open subsets $\om\subset \Om$ in order to simplify later applications of this result.

\begin{lemma}\label{lem:cordes_ineq}
For any open set $\om\subseteq\Om$, and for any $u,v\in H^2(\om)$, writing $w\coloneqq u-v$, the following inequalities hold pointwise a.e. in $\om$
\begin{subequations}\label{eq:cordes_ineq}
\begin{align}
\abs{F_\gamma[u] - F_\gamma[v] - \LL (u-v) }&\leq \sqrt{1-\nu}\sqrt{|\nabla^2 w|^2+2\lambda|\nabla w|^2+\lambda^2|w|^2},\label{eq:cordes_ineq1}\\
|F_\gamma[u]-F_\gamma[v]|&\leq \big(1+\sqrt{d+1}\big)\sqrt{|\nabla^2w|^2+2\lambda|\nabla w|^2+\lambda^2|w|^2}.\label{eq:cordes_ineq2}
\end{align}
\end{subequations}
\end{lemma}
\begin{proof}
For arbitrary bounded sets of real numbers $\{X^{\alpha\beta}\}_{(\alpha,\beta)\in\sA\times\sB}$ and $\{Y^{\alpha\beta}\}_{(\alpha,\beta)\in\sA\times\sB}$, it is easy to see that
$$
\left\lvert\inf_{\alpha\in\sA}\sup_{\beta\in\sB} X^{\alpha\beta}-\inf_{\alpha\in\sA}\sup_{\beta\in\sB} Y^{\alpha\beta} \right\rvert\leq \sup_{(\alpha,\beta)\in\sA\times\sB} \abs{X^{\alpha\beta}-Y^{\alpha\beta}}.
$$
The proof of~\eqref{eq:cordes_ineq1} then follows the same arguments as in~\cite[Lemma~1]{SS14}. The inequality~\eqref{eq:cordes_ineq2} is then obtained from~\eqref{eq:cordes_ineq1} by adding and subtracting $\LL w$ and applying the triangle inequality, along with the Cauchy--Schwarz inequality $\abs{\LL w}\leq \sqrt{d+1} \sqrt{\abs{\nabla^2 w}^2+\lambda^2\abs{w}^2}$.
\end{proof}

Let $\norm{\cdot}_{H^2(\Om)}$ denote the $H^2$-norm of functions in $H^2(\Om)$, defined by 
$$
\begin{aligned}
\norm{v}_{H^2(\Om)}^2\coloneqq \int_\Om\left[\abs{\nabla^2 v}^2+\abs{\nabla v}^2+\abs{v}^2 \right] &&& \forall v \in H^2(\Om).
\end{aligned}
$$
It is well-known that the convexity of $\Omega$ implies that the operator $\LL $ is bijective between $H^2(\Om)\cap H^1_0(\Om)$ and $L^2(\Om)$.
Furthermore, there exists a positive constant $C_{\dim,\diam\Om}$ depending only on $\dim$ and $\diam \Om$, the diameter of $\Om$, such that, for any $\lambda \geq 0$,
\begin{equation}\label{eq:MirandaTalenti}
\begin{aligned}
\frac{1}{C^2_{\dim,\diam\Om}}\norm{v}_{H^2(\Om)}^2 \leq \int_\Om \left[\abs{\nabla^2 v}^2 + 2\lambda \abs{\nabla v}^2 + \lambda^2 \abs{v}^2 \right] \leq \int_\Om \abs{\LL v}^2  &&& \forall v \in H^2(\Om)\cap H^1_0(\Om),
\end{aligned}
\end{equation}
where the first inequality is shown by the Poincar\'e inequality for functions in $H^1_0(\Om)$ and the identity $\int_\Om \abs{\nabla v}^2 = - \int_\Om v \Delta v $ for all $v\in H^2(\Om)\cap H^1_0(\Om)$, and the second inequality follows from the Miranda--Talenti inequality, see e.g.~\cite{MaugeriPalagachevSoftova00,SS13,SS14}.
We now show that there exists a unique strong solution in $H^2(\Om)\cap H^1_0(\Om)$ of the Isaacs equation~\eqref{eq:isaacs_pde} on convex domains under the Cordes condition, which generalises the well-posedness result for HJB equations of~\cite[Theorem~3]{SS14}.

\begin{theorem}[Existence and uniqueness of a strong solution]\label{thm:well_posedness}
There exists a unique $u\in H^2(\Om)\cap H^1_0(\Om)$ that solves $F[u]=0$ pointwise a.e.\ in $\Om$, and, equivalently, that solves $\Fg[u]=0$ pointwise a.e.\ in $\Om$.
\end{theorem}
\begin{proof}
The proof of Theorem~\ref{thm:well_posedness} follows the same arguments as in~\cite[Theorem~3]{SS14}, although we give here the details for completeness.
Let $A\colon H^2(\Om)\cap H^1_0(\Om)\times H^2(\Om)\cap H^1_0(\Om)\tends \R$ be defined by
\begin{equation}
\begin{aligned}
A(w;v)=\int_\Om \Fg[w]\LL v &&& \forall w,\,v \in H^2(\Om)\cap H^1_0(\Om).
\end{aligned}
\end{equation}
We infer from the bijectivity of the operator $\LL \colon H^2(\Om)\cap H^1_0(\Om) \tends L^2(\Om)$ and from Lemma~\ref{lem:subsupersolutions} that $u\in H^2(\Om)\cap H^1_0(\Om)$ solves $F[u]=0$ a.e.\ in $\Om$, and equivalently $\Fg[u]=0$ a.e.\ in $\Om$, if and only if $A(u;v)=0$ for all $v\in H^2(\Om)\cap H^1_0(\Om)$.
It is easy to see from~\eqref{eq:cordes_ineq2} that $A(\cdot;\cdot)$ is bounded and also Lipschitz continuous, i.e.\ that $\abs{A(w;v)-A(z;v)}\leq C \norm{w-z}_{H^2(\Om)}\norm{v}_{H^2(\Om)}$ for all $w$, $z$, $v\in H^2(\Om)\cap H^1_0(\Om)$ for some constant $C$.
We also claim that $A(\cdot;\cdot)$ strongly monotone on $H^2(\Om)\cap H^1_0(\Om)$, which will then imply that there exists a unique $u\in H^2(\Om)\cap H^1_0(\Om)$ that solves $A(u;v)=0$ for all $v\in H^2(\Om)\cap H^1_0(\Om)$ as a result of the Browder--Minty Theorem (see e.g.\ the textbook~\cite{Ciarlet2013}). 
To show strong monotonicity, let $w$, $v\in H^2(\Om)\cap H^1_0(\Om)$ be arbitrary and set $z\coloneqq w - v$; then, by addition and subtraction, we find that
\begin{equation}\label{eq:strong_monotonicity}
\begin{split}
A(w;w-v)-A(v;w-v) &=\int_\Om (\Fg[w]-\Fg[v])\LL z = \int_\Om \abs{\LL z}^2 +  \int_\Om (\Fg[w]-\Fg[v]-\LL z) \LL z \\ & \geq \left(1-\sqrt{1-\nu}\right)\int_\Om \abs{\LL z}^2  \geq (1-\sqrt{1-\nu})\Cdiam^{-2}\norm{z}_{H^2(\Om)}^2,
\end{split}
\end{equation}
where the inequalities in the last line follow from~\eqref{eq:cordes_ineq1} and~\eqref{eq:MirandaTalenti}. This shows that $A(\cdot;\cdot)$ is strongly monotone on~$H^2(\Om)\cap H^1_0(\Om)$ and completes the proof.
\end{proof}

As mentioned at the beginning of this section, this analysis in this section does not make use of the assumption $\dim\in \{2,3\}$, and nor does it require $\Omega$ to be polytopal. Therefore, Theorem~\ref{thm:well_posedness} holds for general bounded convex domains in arbitrary dimensions.

\section{Setting and notation}\label{sec:notation}

For a Lebesgue measurable set $\om \subset \mathbb R^{\dim}$, let $\abs{\om}$ denote its Lebesgue measure, and let $\diam\om$ denote its diameter. The $L^2$-norm of functions over $\omega$ is denoted by $\norm{\cdot}_{\om}$.
Let~$\cT$ be a finite conforming partition of $\Om$ into closed simplices, and let $\shapereg$ denote its shape-regularity parameter defined by
\begin{equation}
\shapereg \coloneqq \max_{K\in\cT} \frac{\diam K}{\rho_K},
\end{equation}
where $\rho_K$ is the diameter of the largest ball inscribed in the element $K$.
In the following, for real numbers $a$ and $b$, we write $a\lesssim b$ if there exists a constant $C$ such that $a\leq C b$, where $C$ depends only on the dimension $\dim$, the domain $\Om$, on $\shapereg$ and on the polynomial degrees $p$ and $q$ defined below, but is otherwise independent of all other quantities. We write $a\eqsim b$ if and only if $a\lesssim b$ and $b\lesssim a$.
Let $\cF$ denote the set of $\dim-1$ dimensional closed faces of the mesh, and let $\cF^I$ and $\cF^B$ denote the subsets of interior faces and boundary faces, respectively.
For each face $F\in\cF$, we consider a fixed choice of unit normal $\bn_F$.
If $F$ is a boundary face then we choose $\bn_F$ to be the unit outward normal to~$\Om$.
To alleviate the notation, we shall usually drop the subscript and simply write $\bn$ when there is no possibility of confusion.
For each $K\in \cT$, we define $h_K \coloneqq  \abs{K}^{\frac{1}{\dim}}$, and note that up to constants depending only on $\dim$ and on~$\shapereg$, we have $h_K \eqsim \diam(K)$.
For each face $F\in\cF$, let $h_F \coloneqq  \left(\mathcal{H}^{\dim-1}(F)\right)^{\frac{1}{\dim-1}}$, where $\mathcal{H}^{\dim-1}$ denotes the $(d-1)$-dimensional Hausdorff measure.
Similarly, we have $h_F\eqsim \diam(F)$ and $h_K\eqsim h_F$ for any element $K\in\cT$ and any face $F\in\cF$ contained in $K$, with constants in the equivalence depending only on $\shapereg$ and on $\dim$.
Let the global mesh-size function $\hT\colon \overline{\Om}\tends\R$ be defined by $\hT|_{K^\circ}=h_K$ for each $K\in\cT$, where $K^\circ$ denotes the interior of $K$, and $\hT|_F=h_F$ for each $F\in\cF$. 
The function $\hT$ is uniformly bounded in~$\Om$, and is only defined up to sets of zero $\mathcal{H}^{d-1}$-measure, which is sufficient for our purposes since $\hT$ only appears in integrals over sets of dimensions~$d-1$ and~$d$. The motivation for this particular definition of $\hT$ can be found in the analysis of adaptive methods, see~\cite{KaweckiSmears20adapt} for further details. For the purposes of this work, it is of course possible to consider common alternative definitions of $\hT$ that are equivalent up to constants depending on shape-regularity of the mesh.

\paragraph{{\bfseries Integration.}} It will be frequently convenient to use a shorthand notation for integrals over collections of elements and faces of the meshes. For any subcollection of elements $\mathcal{E}\subset\cT$, we shall write $\int_{\mathcal{E}}  \coloneqq  \sum_{K\in\mathcal{E}}\int_E$ where the measure of integration is the Lebesgue measure on $\mathbb R^{\dim}$. Likewise, if $\mathcal{G}\subset \cF$, we write $\int_{\mathcal{G}} \coloneqq  \sum_{F\in\mathcal{G}}\int_F$, where the measure of integration is the $(\dim-1)$-dimensional Hausdorff measure on $\mathbb R^{\dim}$. We do not indicate the measure of integration as there is no possibility of confusion.

\paragraph{{\bfseries Partial derivatives.}}
In order to unify and generalise the notions of weak derivatives of Sobolev regular functions and the notion of piecewise derivatives of functions from the finite element spaces, we define notions of gradients and Hessians of functions for certain classes of functions of bounded variation.
Let $BV(\Om)$ denote the space of real-valued functions of bounded variation on $\Omega$, see~\cite{AmbrosioFuscoPallara00,EvansGariepy2015} for precise definitions.
Recall that $BV(\Om)$ is a Banach space equipped with the norm $\norm{v}_{BV(\Om)}\coloneqq\norm{v}_{L^1(\Om)}+\abs{Dv}(\Om)$, where $\abs{Dv}(\Om)$ denotes the total variation of its distributional derivative $Dv$ over $\Omega$, defined by $\abs{Dv}(\Om)\coloneqq \sup\left\{\int_\Om v \Div \bphi\colon \bphi\in C^\infty_0(\Om;\R^\dim), \norm{\bphi}_{C(\overline{\Om};\R^\dim)}=1 \right\}$.

For any $v\in BV(\Omega)$, the distributional derivative $Dv$ can be identified with a Radon measure on $\Omega$ that can be decomposed into the sum of an absolutely continuous part with respect to Lebesgue measure, and a singular part~\cite[p.~196]{EvansGariepy2015}.
Let $\nabla v\in L^1(\Om;\R^\dim)$ denote the (vector) density of the absolutely continuous part of $D v$ with respect to Lebesgue measure.
Following \cite{FonsecaLeoniParoni05}, for functions $v\in BV(\Om)$ such that $\Npw v \in BV(\Om;\R^\dim)$, we define $\Dpw v$ as the density of the absolutely continuous part of $D(\nabla v)$ the distributional derivative of $\nabla v$; in particular, 
\begin{equation}\label{eq:Hessian_notation}
\begin{aligned}
\nabla^2 v \coloneqq \nabla(\nabla v) \in L^1(\Om;\R^{\dim\times\dim}), &&&
(\nabla^2 v)_{ij} \coloneqq \nabla_{x_j} (\nabla_{x_i} v) \quad \forall i,\, j\in \{1,\dots,\dim\}.
\end{aligned}
\end{equation}
The Laplacian $\Delta v$ is defined as the matrix trace of $\nabla^2 v$.
Note that $\Dpw v$ is defined in terms of $D(\nabla v)$ and not $D^2v$ the second distributional derivative of $v$ since in general $D^2 v$ is not necessarily a Radon measure.
The definitions above unify the concepts of weak derivatives of functions in Sobolev spaces over $\Omega$ and of piecewise derivatives of functions from the DG and $C^0$-IP finite element spaces defined shortly below. 
Indeed, it is easy to see that the above definition of $\Npw v$ coincides with the weak gradient of $v$ if $v\in W^{1,1}(\Om)$ and that $\Dpw v$ coincides with the weak Hessian of $v$ if $v\in W^{2,1}(\Om)$.
Moreover, for functions that are piecewise smooth over the mesh $\cT$, such as functions from the finite element spaces defined below, it is easy to see that the gradient and Hessian as defined above coincide with the piecewise gradient and Hessian over elements of the mesh.
The nonlinear Isaacs operators $F$ from~\eqref{eq:F_def} and $\Fg$ from~\eqref{eq:Fg_def} are then naturally extended to all functions $v\in BV(\Om)$ such that $\Npw v \in BV(\Om;\R^\dim)$.

\paragraph{{\bfseries Jump, average and tangential differential operators on faces.}}
There is a bounded trace operator $\tau_{\p K} \colon BV(K) \tends L^1(\p K)$ for each $K\in\cT$, see e.g.\ \cite{EvansGariepy2015}.
 It follows that a function $v\in BV(\Om)$, once restricted to an element $K\in \cT$, has a trace $\tau_{\p K} v\coloneqq  \tau_{\p K} (v|_K) \in  L^1(\p K)$. In general, if $F$ is an interior face of the mesh, i.e.\ $F=\p K\cap \p K^\prime$ for $K$, $K^\prime\in \cT$, then $\tau_{\p K} v|_F \neq \tau_{\p K^\prime} v|_F $, i.e.\ traces from different elements do not necessarily agree on a common face.
For $v\in BV(\Om)$, we define the jump $\jump{v}_F\in L^1(F)$ and average of $\avg{v}_F\in L^1(F)$ for each $F\in\cF$ by 
\begin{equation}\label{eq:jumpavg}
\begin{aligned}
  \avg{v}_F&\coloneqq  \frac{1}{2}\left(\tau_{\p K}v|_F+\tau_{\p K'}v|_F\right), & \jump{v}_F & \coloneqq  \tau_{\p K}v|_F-\tau_{\p K'}v|_F, &\forall F\in\cF^I,\\
  \avg{v}_F &\coloneqq  \tau_{\p K} v|_F &  \jump{v}_F & \coloneqq  \tau_{\p K} v|_F & \forall F\in\cF^B,
\end{aligned}  
\end{equation}
where, in the case $F\in\cF^I$, the elements $K$ and $K^\prime \in \cT$ are labelled such that the chosen unit normal $\bn_F$ is the outward normal to $K$ on $F$ and the inward normal to $K'$ on $F$, and where the trace operators $\tau_{\p K}$ and $\tau_{\p K^\prime}$ are applied to the restrictions of the function $v$ to $K$ and $K^\prime$, respectively.
The jump and average operators are further extended to vector fields in $BV(\Om;\mathbb R^\dim)$ componentwise.
Although the sign of $\jump{v}_F$ depends on the choice of $\bn_F$, in subsequent expressions the jumps will appear either under absolute value signs or in products with $\bn_F$, so that the overall resulting expression is uniquely defined and independent of the choice of $\bn_F$.
When no confusion is possible, we drop the subscripts and simply write $\avg{\cdot}$ and~$\jump{\cdot}$. 

For $F\in \cF$, let $\nabla_T$ denote the tangential (surface) gradient operator, and let $\Delta_T$ denote the tangential Laplacian, which are defined for all sufficiently smooth functions on $F$.
We do not indicate the dependence of these operators on $F$ in order to alleviate the notation, as it will be clear from the context.

\paragraph{{\bfseries Finite element spaces.}}
For a fixed choice of polynomial degree $p\geq 2$, let the finite element spaces~$\Vt^s$, $s\in\{0,1\}$, be defined by
\begin{equation}\label{sec:fem_spaces}
\begin{aligned}
  \Vt^0 &\coloneqq  \{\vt\in L^2(\Om):\vt|_K\in \Pp\;\forall K\in\cT\},
   &\Vt^1&\coloneqq   \Vt^0 \cap H^1_0(\Om),
\end{aligned}  
\end{equation}
where $\Pp$ denotes the space of polynomials of total degree at most $p$. 
The condition $p\ge2$ is required due to the fact that we seek approximations in $H^2$-type norms, thus requiring at least piecewise quadratic polynomials to approximate the Hessian of the true solution.
The spaces $\Vt^0$ and $\Vt^1$ correspond to DG and $C^0$-IP spaces on $\cT$, respectively.
It is clear that if $v\in  \Vt^s$, $s\in\{0,1\}$, then $v\in BV(\Om)$ and that $\nabla v$, as defined above, coincides with the piecewise gradient of $v$ over the elements of the mesh $\cT$. It then follows that $\nabla v \in BV(\Om;\R^\dim)$ and that the Hessian $\Dpw v$ defined above coincides with the piecewise Hessian of $v$ over the elements of the mesh.

The spaces $\Vt^s$, $s\in\{0,1\}$, are equipped with the norm $\normt{\cdot}$ and jump seminorm $\absJ{\cdot}$ defined by
\begin{equation}\label{eq:norm_def}
\begin{aligned}
\normt{\vt}^2\coloneqq  \int_{\Om}\left[ \abs{\Dpw \vt}^2+\abs{\Npw \vt}^2 + \abs{\vt}^2\right]+\absJ{v}^2, &&& \absJ{v}^2\coloneqq \int_{\cF^I}\hT^{-1}\abs{\jump{\Npw\vt}}^2+\int_{\cF}\hT^{-3}\abs{\jump{\vt}}^2,
\end{aligned}
\end{equation}
for all $\vt \in \Vt^s$.
Although  $\Vt^0$ and $\Vt^1$ are equipped with the same norm and jump seminorm, it is clear that for any $v\in \Vt^1 \subset H^1_0(\Om)$, the  last term in the right-hand side~\eqref{eq:norm_def} involving jumps over mesh faces vanishes and that the terms involving jumps of first derivatives over internal mesh faces can be simplified to merely jumps of normal derivatives.
However, these simplifications do not play any particular role in the subsequent analysis and do not need to be considered further.
The norm~$\normt{\cdot}$ and jump semi-norm $\absJ{\cdot}$ extend to the sum space $\Vt^s + \Hd$, where $H=H^2(\Om)\cap H^1_0(\Om)$.
Note that for general $v\in \Vt^s+\Hd$, we have $\absJ{v}=0$ if and only if $v\in \Hd$.

\paragraph{{\bfseries Poincar\'e--Friedrichs inequality.}} Although we consider here the norm~$\normt{\cdot}$ given in~\eqref{eq:norm_def}, our results are by no means specific to this choice of norm. This is a consequence of the following second-order Poincar\'e--Friedrichs inequality for functions in $\Vt^s$, which shows that the norm $\normt{\cdot}$ is equivalent to  other $H^2$-type norms.

\begin{theorem}[Poincar\'e--Friedrichs inequality]\label{thm:poincare}
There exists a constant $\CPF$ depending only on $\dim$, $\shapereg$, $p$, and on $\diam \Omega$ such that
\begin{equation}\label{eq:Poincare}
\begin{aligned}
\normt{\vt}\leq C_{\mathrm{PF}}\left( \int_\Om \abs{\nabla^2 \vt}^2 + \absJ{\vt}^2 \right)^{\frac{1}{2}} &&&\forall \vt \in \Vt^s, \quad\forall s\in\{0,1\}.
\end{aligned}
\end{equation}
\end{theorem}
\begin{proof}
We divide the proof into two steps, treating first the case $s=1$ followed by the more general case $s=0$. We note that in both cases, it is enough to show that the lower order terms in~\eqref{eq:norm_def} are bounded by the right-hand side of~\eqref{eq:Poincare}.

\emph{Step~1.} Suppose that $s=1$, and let $\vt\in \Vt^1$ be arbitrary. Then, integration-by-parts and an inverse inequality yield
\[
\int_\Om \abs{\nabla \vt}^2 = - \int_\Om \vt \Delta \vt + \int_{\cF^I} \vt \jump{\nabla \vt\cdot \bn} \lesssim \left(\int_\Om \abs{\nabla^2 \vt}^2 +  \absJ{\vt}^2\right)^{\frac{1}{2}}\norm{\vt}_\Om ,
\]
where the constant in the inequality above depends only on $\dim$, $\shapereg$ and $p$. Since $\Vt^1\subset H^1_0(\Om)$, we have $\norm{\vt}_\Om \leq C_{\diam\Om} \norm{\nabla \vt}_\Om$ with a constant $C_{\diam\Om}$ depending only on $\diam\Omega$, from which~\eqref{eq:Poincare} for $s=1$ follows immediately.

\emph{Step~2.} Suppose now that $s=0$ and let $\vt\in\Vt^0$ be arbitrary. We use the $H^1_0$-enrichment operators from \cite{KarakashianPascal03,HoustonSchotzauWihler07}. In particular, there exists a linear operator $E_1\colon \Vt^0\tends \Vt^1$ such that
\begin{equation}\label{eq:enrichment_operator_1}
\begin{aligned}
\sum_{m=0}^{2}\int_K\hT^{2m-4}\abs{\nabla^m(\vt-E_1\vt) }^2 \lesssim \int_{\cF_K} \hT^{-3}\abs{\jump{\vt}}^2, &&& \forall K\in\cT,\;\forall \vt\in\Vt^0,
\end{aligned}
\end{equation}
where $\cF_K\coloneqq \{F\in\cF,\,F\cap K\neq \emptyset\}$ is the set of faces neighbouring the element $K$. 
The constant in~\eqref{eq:enrichment_operator_1} depends only on $\dim$, $\shapereg$ and $p$, but not on $\Omega$.
In particular, the bound in~\eqref{eq:enrichment_operator_1} for $m=1$ is a consequence of~\cite[Theorem~2.2]{KarakashianPascal03}, and the cases $m\in\{0,2\}$ are shown in a similar manner by  scaling arguments.
We then infer from the triangle inequality, the trace inequality and~\eqref{eq:enrichment_operator_1} that $\absJ{E_1\vt}\leq\absJ{\vt}+\absJ{\vt-E_1\vt}\lesssim \absJ{\vt}$ again with a constant depending only on $\dim$, $\shapereg$ and $p$. 
Therefore, using the triangle inequality and~\eqref{eq:Poincare} for $s=1$ and~\eqref{eq:enrichment_operator_1}, we find that
\begin{multline*}
\norm{\nabla \vt}_\Om^2 \lesssim \norm{\nabla E_1\vt}_\Om^2 + \norm{\nabla(\vt-E_1\vt)}_\Om^2 \\ \lesssim \int_\Om \abs{\nabla^2E_1 \vt}^2+\absJ{E_1\vt}^2 + \absJ{\vt}^2 \lesssim \int_\Om \abs{\nabla^2 \vt}^2 + \absJ{\vt}^2,
\end{multline*}
with a constant depending only on $\dim$, $p$, $\shapereg$ and $\diam\Om$. We then obtain~\eqref{eq:Poincare} upon recalling the inequality $\norm{\vt}^2_\Om\lesssim \norm{\nabla \vt}^2_\Om+\int_{\cF}\hT^{-1}\abs{\jump{\vt}}^2$ with a constant depending only on $\dim$, $\shapereg$, and $\diam\Om$, see e.g.~\cite{Brenner03}.
\end{proof}

In the subsequent analysis, we occasionally use the $\lambda$-weighted seminorm $\absL{\cdot}\colon\Vt^s\tends \R$ defined by 
\begin{equation}\label{eq:lambda_seminorm}
\begin{aligned}
\absL{\vt}^2 \coloneqq \int_\Om \left[ \abs{\nabla^2 \vt}^2+2\lambda \abs{\nabla \vt}^2+\lambda^2\abs{\vt}^2\right] &&&\forall\vt \in \Vt^s.
\end{aligned}
\end{equation}
In general $\absL{\cdot}$ is only a seminorm for $\lambda\geq 0$, but is a norm if $\lambda>0$.
It is clear that $\absL{\vt}^2+\absJ{\vt}^2\leq c_{\lambda}^2 \normt{\vt}^2$ with constant $c_\lambda=\max\{1,\sqrt{2\lambda},\lambda\}$ for all $\vt\in\Vt^s$.
Theorem~\ref{thm:poincare} implies a converse bound, namely that $\normt{\vt}^2\lesssim \absL{\vt}^2+\absJ{\vt}^2$ for all $\vt\in\Vt^s$.

\section{General framework for \emph{a priori} and \emph{a posteriori} error analysis}\label{sec:framework}

We now present a general framework for the \emph{a priori} and \emph{a posteriori} error analysis of a broad range of numerical methods.
We start by showing that the \emph{a posteriori} error analysis is essentially determined only by the approximation spaces, and is otherwise independent of the choice of numerical methods.
For this reason, we present the \emph{a posteriori} error bound before discussing numerical discretizations of~\eqref{eq:isaacs_pde}.

\subsection{A posteriori error bound}\label{sec:abtract_apost}

Our first main result is an~\emph{a posteriori} error bound, where we prove reliability and local efficiency of residual-type error estimators.
The analysis hinges on the following Lemma, which shows that the jump seminorm $\absJ{\cdot}$ defined in~\eqref{eq:norm_def} controls the distance of functions $\Vt^s$ from $H^2(\Om)\cap H^1_0(\Om)$.
See also~\cite{S18} for related results that are explicit in the polynomial degree on more general meshes, and see also the concluding remarks in~\cite{BrennerSung2019}.

\begin{lemma}[$H^2(\Om)\cap H^1_0(\Om)$-\textbf{Approximation}]\label{lem:enrich}
There exists a linear operator $E_\cT:\Vt^0\to H^2(\Om)\cap H^1_0(\Om)$ such that
\begin{equation}\label{eq:enrichment_operator}
\begin{aligned}
\sum_{m=0}^2 \int_\Om \hT^{2m-4} \abs{\nabla^m(\vt - E_{\cT}\vt)}^2 \lesssim \absJ{\vt}^2 &&&\forall \vt\in\Vt^0.
\end{aligned}
\end{equation}
\end{lemma}
\begin{proof}
We recall the operator $E_1\colon\Vt^0\tends\Vt^1$ used in the proof of Theorem~\ref{thm:poincare} above.
Furthermore, in~\cite{NeilanWu19} (for $\dim=2$ and $p\geq 2$ or $\dim=3$ and $2\leq p \leq 3$) and in~\cite{BrennerSung2019} (for $\dim\in\{2,3\}$ and $p\geq 2$) it is shown that there exists a linear operator $E_2\colon \Vt^1\tends H^2(\Om)\cap H^1_0(\Om)$ such that
\begin{equation}\label{eq:enrichment_operator_2}
\begin{aligned}
\sum_{m=0}^{2}\int_K \hT^{2m-4}\abs{\nabla^m(\widetilde{v}_{\cT}-E_2\widetilde{v}_{\cT})}^2 \lesssim \int_{\cF^I_K} \hT^{-1}\abs{\jump{\Npw \widetilde{v}_{\cT} \cdot \bn}}^2 &&&\forall K\in \cT,\;\forall \widetilde{v}_{\cT}\in\Vt^1,
\end{aligned}
\end{equation}
where $\cF^I_K\coloneqq \cF_K\cap \cF^I$ is the set of interior faces adjacent to $K$, see~\cite[Lemma~3]{NeilanWu19} and \cite{BrennerSung2019}. Then, we define the operator $E_{\cT}$ as the composition of the operators $E_1$ and $E_2$, i.e.\ $E_{\cT}\coloneqq E_2 E_1 $, and~\eqref{eq:enrichment_operator} is obtained by applying the triangle inequality to $\vt-E_{\cT}\vt=\vt-E_1\vt+E_1\vt-E_2(E_1 \vt)$ and applying the bounds~\eqref{eq:enrichment_operator_1} and~\eqref{eq:enrichment_operator_2} with summation over all elements of the mesh.
\end{proof}

The primary use of Lemma~\ref{lem:enrich} for our purposes is the implication that
\begin{equation}\label{eq:C_E_constant}
\begin{aligned}
\inf_{w\in H^2(\Om)\cap H^1_0(\Om)} \normt{\vt-w} \lesssim  \absJ{\vt} &&& \forall \vt\in\Vt^s,\;\quad\forall s\in\{0,1\}, 
\end{aligned}
\end{equation}
where the constant in the inequality above depends possibly on $\dim$, on $\shapereg$, on $p$ and on $\Omega$; see also Remark~\ref{rem:caveat} below for further discussion of the constants.
We now introduce the residual-type error estimators that form the basis of the \emph{a posteriori} error analysis.
For any $\vt\in\Vt^s$, $s\in\{0,1\}$, let the elementwise error estimators $\{\eT(\vt,K)\}_{K\in\cT}$ and global error estimator $\eT(\vt)$ be defined by
\begin{subequations}\label{eq:error_estimators}
\begin{align}
[\eT(\vt,K)]^2 &\coloneqq \int_K \abs{\Fg[\vt]}^2 + \sum_{\substack{F\in\cF^I\\ F\subset\p K}}\int_F\delta_F \hT^{-1}\abs{\jump{\nabla \vt}}^2 + \sum_{\substack{F\in\cF \\ F\subset\p K}}\int_F \delta_F\hT^{-3}\abs{\jump{\vt}}^2,
\\ [\eT(\vt)]^2 &\coloneqq \sum_{K\in\cT} [\eT(\vt,K)]^2,
\end{align}
\end{subequations}
with weights $\delta_F \coloneqq 1/2$ if $F\in\cF^I$ and $\delta_F\coloneqq 1$ if $F\in\cF^B$.
Recall that the expression $\Fg[\vt]$ is computed using the notion of gradients and Hessians of $\vt$ as defined in Section~\ref{sec:notation}, which, for functions from the finite element spaces, coincide with the notions of piecewise gradients and Hessians, respectively.
For the special case $s=1$, we note that the term involving the jumps $\jump{\vt}$ vanishes identically and thus may be dropped, and that the term involving jumps of gradients can be simplified to the jumps in the normal component of the gradients.
However these simplifications have no special consequence in the results below.
In practice, one may consider a number of variants of the estimators in~\eqref{eq:error_estimators}, e.g.\ including various weightings of the different terms; we employ the above choice of estimators for simplicity of presentation.

For each element $K\in\cT$, we define $\norm{\cdot}_{\cT,K}\colon \Vt^s+\Hd\tends \R$ the localization of the norm $\norm{\cdot}_{\cT}$ to $K$ by
\begin{equation}
\norm{v}_{\cT,K}^2 \coloneqq \int_K \left[ \abs{\nabla^2 v}^2+\abs{\nabla v}^2+\abs{v}^2 \right] 
+ \sum_{\substack{F\in\cF^I \\ F\subset\p K}}\int_F \delta_F \hT^{-1}\abs{\jump{\nabla v}}^2 + \sum_{\substack{F\in\cF \\ F\subset\p K}}\int_F \delta_F\hT^{-3}\abs{\jump{v}}^2.
\end{equation}
For any $v\in \Vt^s+\Hd$ there holds $\norm{v}_{\cT}^2 = \sum_{K\in\cT}\norm{v}_{\cT,K}^2$.

We now present an \emph{a posteriori} error bound for arbitrary functions from the approximation space, and not only the numerical solution.
Recall that $u\in H^2(\Om)\cap H^1_0(\Om)$ denotes the unique solution of~$F[u]=0$ and equivalently of $\Fg[u]=0$ pointwise a.e.\ in $\Om$, see Theorem~\ref{thm:well_posedness}.

\begin{theorem}[A posteriori error bound]\label{thm:aposteriori}
There exists a positive constant $C_{\mathrm{rel}}$ depending only on $\dim$, $\shapereg$, $p$, $\lambda$, $\nu$ and $\Om$, such that, for any $s\in\{0,1\}$,
\begin{equation}\label{eq:aposteriori_reliability}
\begin{aligned}
\normt{u-\vt} \leq C_{\mathrm{rel}} \eT(\vt) &&&\forall \vt\in\Vt^s.
\end{aligned}
\end{equation}
There exists a positive constant $C_{\mathrm{eff},\mathrm{loc}}$ depending only on $\dim$ and $\lambda$, such that
\begin{equation}\label{eq:aposteriori_local_efficiency}
\begin{aligned}
 \eT(\vt,K)  \leq C_{\mathrm{eff},\mathrm{loc}} \norm{u- \vt}_{\cT,K} &&&\forall K\in\cT,\;\forall \vt\in\Vt^s.
 \end{aligned}
\end{equation} 
There exists a positive constant $C_{\mathrm{eff},\mathrm{glob}}$ depending only on $\dim$ and $\lambda$, such that
\begin{equation}\label{eq:aposteriori_global_efficiency}
\begin{aligned}
 \eT(\vt)  \leq C_{\mathrm{eff},\mathrm{glob}} \norm{u - \vt}_{\cT} &&& \forall \vt\in\Vt^s.
 \end{aligned}
\end{equation}
\end{theorem}
\begin{proof}
Let $\vt\in \Vt^s$ and $w\in H^2(\Om)\cap H^1_0(\Om)$ be arbitrary functions. Then, recalling~\eqref{eq:strong_monotonicity} we see that
\[
\begin{aligned}
(1-\sqrt{1-\nu})\norm{\LL(u-w)}_\Om^2 &\leq \int_{\Om}\Fg[w]\LL(u-w) \\  &\leq (\norm{\Fg[\vt]}_\Om+\norm{\Fg[w]-\Fg[\vt]}_{\Om})\norm{\LL(u-w)}_\Om.
\end{aligned}
\]
Then, using the fact that $\normt{u-w}=\norm{u-w}_{H^2(\Om)}$, and by combining the above inequality with~\eqref{eq:MirandaTalenti} and the Lipschitz continuity bound of $\Fg$ in~\eqref{eq:cordes_ineq2}, we find that
\begin{equation}\label{eq:apost_1}
\begin{split}
\normt{u-\vt} & \leq \norm{u-w}_{H^2(\Om)}+\normt{w-\vt} 
\\ &\leq C_{\dim,\diam\Om}\norm{\LL(u-w)}_{\Om} + \normt{\vt-w} 
\\ & \leq C_{\dim,\diam\Om} c_\nu \left(\norm{\Fg[\vt]}_\Om +\norm{\Fg[w]-\Fg[\vt]}_\Om \right) + \norm{\vt-w}_{\cT}
\\ &\leq C_{\dim,\diam\Om} c_\nu \norm{\Fg[\vt]}_\Om + (1+ C_{\dim,\diam\Om} c_\nu c_\lambda(1+\sqrt{d+1}))\normt{\vt-w},
\end{split}
\end{equation}
with $c_\nu=(1-\sqrt{1-\nu})^{-1}$, and $c_\lambda=\max\{1,\sqrt{2\lambda},\lambda\}$.
Since the function $w$ in~\eqref{eq:apost_1} is arbitrary, we may take the infimum over all $w\in H^2(\Om)\cap H^1_0(\Om)$ and apply~\eqref{eq:C_E_constant} to obtain 
\[
\norm{u-\vt}_\Om \leq C_2\left(\norm{\Fg[\vt]}_\Om+ \absJ{\vt}\right)\leq C_{\mathrm{rel}} \eT(\vt),
\]
for some constants $C_2$ and $C_{\mathrm{rel}}$ that depend possibly on $\dim$, $\shapereg$, $p$, $\nu$, $\lambda$ and $\Om$, which proves~\eqref{eq:aposteriori_reliability}.
To prove~\eqref{eq:aposteriori_local_efficiency}, we use Theorem~\ref{thm:well_posedness} which shows that $u\in H^2(\Om)\cap H^1_0(\Om)$ solves $\Fg[u]=0$ pointwise a.e.\ and thus infer that, for all $K\in\cT$,
\[
\begin{split}
\eT(\vt,K) \leq c_\lambda(1+\sqrt{d+1})\norm{\vt-u}_{\cT,K},
\end{split}
\]
with $c_\lambda$ as above, where we have used the Lipschitz bound from~\eqref{eq:cordes_ineq2} to bound $\norm{\Fg[\vt]}_\Om=\norm{\Fg[\vt]-\Fg[u]}_\Om$.
This gives~\eqref{eq:aposteriori_local_efficiency} with $C_{\mathrm{eff,loc}}=c_{\lambda}(1+\sqrt{d+1})$.
We then obtain~\eqref{eq:aposteriori_global_efficiency} from~\eqref{eq:aposteriori_local_efficiency} by taking square powers and summing over all elements of the mesh.
\end{proof}

\begin{remark}
A posteriori error bounds of a similar nature have been shown already in~\cite{Gallistl17,Bleschmidt19,Kawecki19c} for various numerical methods.
However, Theorem~\ref{thm:aposteriori} shows that the a posteriori error bounds are not restricted to any particular numerical method, as the bounds apply to arbitrary piecewise polynomial approximations on $\cT$.
The significance for computational practice is then that the error estimators are reliable and efficient even for inexactly computed numerical solutions, obtained from iterative solvers for the nonlinear discrete problem.
Furthermore, Theorem~\ref{thm:aposteriori} can be applied to a wide range of approximations using various finite element spaces, such as Morley or Hermite elements, and, up to substituting $\cT$ for a submesh, macro-elements such as the Hsieh--Clough--Tocher element~\cite{Ciarlet02}. Naturally, there may be some simplifications that can be made in the estimators when taking their restrictions to subspaces of $\Vt^0$ with higher regularity.
The fact that Theorem~\ref{thm:aposteriori} holds for arbitrary $\vt \in \Vt^s$ presents some substantial differences with the case of the usual residual-based error estimators for weak solutions of divergence form elliptic problems.
Recall that for divergence form problems, the derivation of the upper bound relies on some form of Galerkin orthogonality satisfied by the numerical solution~\cite{Verfurth13,CarstensenGudiJensen09}. This is due to the issue of localizing and bounding the negative-order Sobolev norm of the residual, see~\cite{BlechtaMalekVohralik2020} for further details.
By comparison, in the present setting the residual is in $L^2$, so the residual norm localizes trivially.
\end{remark}

Theorem~\ref{thm:aposteriori} shows that the residual-type estimators in~\eqref{eq:error_estimators} are reliable and locally efficient. In~\cite{KaweckiSmears20adapt}, we use these estimators to construct and prove convergence of adaptive DG and $C^0$-IP methods for the problem at hand. 
Note that the estimators in the present setting have some notable differences with residual estimators for approximations of divergence form elliptic problems in $H^1$-type norms~\cite{Verfurth13}.
The estimators defined in~\eqref{eq:error_estimators} do not include any weighting of the volume residual terms with positive powers of the mesh-size function $\hT$; this is indeed both natural and optimal as shown by the efficiency bounds~\eqref{eq:aposteriori_local_efficiency} and~\eqref{eq:aposteriori_global_efficiency}.
This has important ramifications for the analysis of adaptive methods~\cite{KaweckiSmears20adapt}.
In comparison to residual estimators for divergence form elliptic problems, here the residual term for the PDE is entirely located on the elements, and the face terms measure only the nonconformity of the approximations.
In consequence, the local efficiency bound~\eqref{eq:aposteriori_local_efficiency} is indeed fully local to an element and to its faces.

The estimators given here are reliable, although it appears harder to make them guaranteed, i.e.\ to obtain a guaranteed upper bound on the error without unknown constants, since this would require determining the constant $C_{\mathrm{rel}}$.
Indeed, the principal difficulty is to determine the constant in~\eqref{eq:C_E_constant} that feeds into $C_{\mathrm{rel}}$.
It appears possible however to obtain a guaranteed and fully computable estimator by replacing the part of the estimator associated to the jumps of function values and gradients over mesh faces by a computable choice of $w\in H^2(\Om)\cap H^1_0(\Om)$ that appears in the proof of Theorem~\ref{thm:aposteriori}, for instance using the approximation constructed in Lemma~\ref{lem:enrich}.
This however appears to be rather involved in practice, so we do not consider it further.
In the numerical experiment of Section~\ref{sec:numexp} below, it is found that in practice the estimators are quite close to the true errors, suggesting that $C_{\mathrm{rel}}$ is close to a value of $1$ for that experiment.

\begin{figure}
\begin{center}
\includegraphics{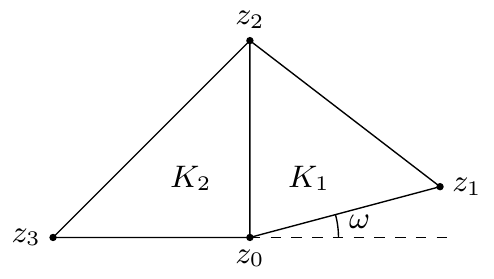}
\caption{[Example of Remark~\ref{rem:caveat}]. A pair of elements $K_1$ and $K_2$ are formed by the vertices $z_0=0$, $z_1=(\cos\omega,\sin\omega)$ for some $\omega\in(0,\pi/2)$, $z_2=(0,1)$ and $z_3=(-1,0)$. Suppose that the lower edges of $K_1$ and $K_2$ are on the boundary $\p\Omega$, so that $z_0$ is a corner (sharp) vertex of $\p\Omega$.
}\label{fig:corner}
\end{center}
\end{figure}

\begin{remark}[Dependence of constants on domain geometry]\label{rem:caveat}
The constant in~\eqref{eq:enrichment_operator} possibly depends on the space dimension $\dim$, the shape-regularity parameter $\shapereg$, the polynomial degree $p$ as may be expected. However,  the constant in the bound~\eqref{eq:enrichment_operator} also depends on constants appearing in the analysis in~\cite{NeilanWu19,BrennerSung2019} that are not robust with respect to the geometry of the boundary $\p\Omega$, as we now explain.
It is enough to consider momentarily $\dim=2$ and $s=1$; then the enrichment operators from~\cite{NeilanWu19,BrennerSung2019} (both labelled here $E_2$ in a slight abuse of notation) both prescribe that the gradient of the $H^2\cap H^1_0$-enrichment approximation must vanish identically at corner points of the boundary (called \emph{sharp} vertices in \cite{NeilanWu19}) see e.g.~\cite[Lemma~2]{NeilanWu19} and~\cite[Section~3.3.1]{BrennerSung2019}.
Supposing that $\vt \in \Vt^1$ is the function to be approximated, and $z$ is a sharp (corner) vertex of $\p\Omega$, then the analysis in the references above involve a bound of the form
\begin{equation}\label{eq:corner_bound_issue}
\abs{\nabla \vt|_K(z)-\nabla E_2 \vt(z)}^2 = \abs{\nabla \vt|_K(z)}^2\leq C_{\sharp} \sum_{F\in\cF^I;z\in F} \int_F \hT^{1-\dim} \abs{\jump{\nabla\vt\cdot\bn}}^2,
\end{equation}
for all elements $K$ sharing the vertex $z$, see~\cite[eq.~(3.11)]{NeilanWu19} and the first displayed equation in~\cite[p.~11]{BrennerSung2019}.
The proof that such a constant exists involves writing $\nabla \vt|_K(z)$ in terms of a local basis formed by tangent vectors of faces. However, the constant $C_{\sharp}$ in~\eqref{eq:corner_bound_issue} generally depends on the angle formed by the tangent vectors and thus on the geometry of $\Omega$, as illustrated by the following example.
Consider a corner vertex $z_0$ associated to a pair of elements $K_1$ and $K_2$ as shown in~Figure~\ref{fig:corner}, and consider a function $\vt$ such that $\vt|_{K_1}(x,y)=y- x\tan\omega$ and $\vt|_{K_2}(x,y)=y$, so that $\vt$ is piecewise affine, continuous on $K_1\cup K_2$, and vanishes on the boundary faces formed by the vertices $z_0$, $z_1$ and $z_3$.
Then, it follows that $\abs{\nabla \vt|_K(z_0)}^2 \geq 1$ for $K\in\{K_1,K_2\}$, whereas  $\int_F \hT^{1-\dim} \abs{\jump{\nabla\vt\cdot\bn}}^2 = \tan^2\omega$ for the interior face $F$ formed by the vertices $z_0$ and $z_2$. Therefore, the constant $C_{\sharp}$ in~\eqref{eq:corner_bound_issue} necessarily satisfies $C_{\sharp}\geq \tan^{-2}\omega$ and thus becomes large for small $\omega$, i.e.\ when $\Omega$ has very nearly flat corners.
Therefore, the claim in~\cite[Thm.~2.1]{BrennerSung2019} that the constants there depend only on the shape~regularity of the meshes appears to have overlooked the dependence on the geometry of the boundary.
In three space dimensions, this geometric dependence also occurs for degrees of freedom on edges belonging to two non-coplanar boundary faces.
\end{remark}

\subsection{Abstract a priori error bound}\label{sec:abstract_apriori}

We now provide a unifying framework for the \emph{a priori} error analysis of a broad family of numerical methods.
Some concrete examples of methods that we have in mind are given in Section~\ref{sec:num_schemes}, which covers a range of different methods proposed in the literature as well as some original variants, see in particular the definition in~\eqref{eq:At_def} and also Remark~\ref{rem:relation_literature} below for further details.
We consider an abstract numerical method of the form: for a chosen $s\in\{0,1\}$, find $u_\cT\in \Vt^s$ such that
\begin{equation}\label{eq:num_scheme}
\At(u_\cT;v_\cT) = 0\quad\forall v_\cT\in \Vt^s,
\end{equation}
for a given nonlinear form $\At(\cdot;\cdot)\colon \Vt^s\times\Vt^s \tends \R$.
We prove a near-best approximation result under abstract assumptions on $\At(\cdot,\cdot)$, which allows for a unified treatment of a range of numerical methods from the literature, and some original methods as well.
First, we assume that the nonlinear form $\At(\cdot;\cdot)$ is linear in its second argument, i.e.\ $\At(\wt;\vt+\delta z_{\cT})=\At(\wt;\vt)+\delta \At(\wt;z_{\cT})$ for all $\vt$, $\wt$ and $z_{\cT} \in \Vt^s$ and $\delta \in \R$.
Next, we make the following three assumptions concerning Lipschitz continuity,  discrete consistency and strong monotonicity.

\emph{Lipschitz continuity.}
The nonlinear form $\At$ is assumed to be Lipschitz continuous, i.e.\ there exists a positive constant $C_{\mathrm{Lip}}$ such that
\begin{gather}\label{eq:discrete_cont}
\begin{aligned}
\abs{\At(w_\cT;v_\cT)-\At(z_\cT;v_\cT)}\leq C_{\mathrm{Lip}}\normt{w_\cT-z_\cT} \normt{v_\cT} &&&\forall \wt,\, z_{\cT}, \, \vt \in \Vt^s.
\end{aligned}\tag{A1}
\end{gather}

\emph{Discrete consistency.}
We assume that there exists a linear operator $L_{\cT} \colon \Vt^s\tends L^2(\Om)$ and positive constants $C_{\mathrm{cons}}$ and $C_{L_{\cT}}$ and such that, for all $\wt,\,\vt\in\Vt^s$, 
\begin{gather}\label{eq:discrete_consistency}
\begin{aligned}
\left\lvert \At(\wt;\vt) - \int_{\Om}\Fg[\wt]L_{\cT} \vt \right\rvert \leq C_{\mathrm{cons}} \absJ{\wt}\normt{\vt}, &&& \norm{L_{\cT} \vt}_\Om \leq C_{L_{\cT}} \normt{\vt}.
\end{aligned}\tag{A2}
\end{gather}
We stress that the jump seminorm $\absJ{\wt}$ appears in the right-hand side of the first inequality.
Therefore, the condition~\eqref{eq:discrete_consistency} requires that $\At(\wt;\vt)$ must include the testing of the nonlinear operator $\Fg$ with $L_{\cT}\vt$ for a test function $\vt\in\Vt^s$, and that any additional terms must vanish whenever $\absJ{\wt}=0$, i.e.\ when the first argument $\wt$ belongs to $H^2(\Om)\cap H^1_0(\Om)$.
The assumption on $L_{\cT}$ is rather general and allows for testing the PDE with a range of choices, although in practice, $L_{\cT}$ is usually chosen with a view towards satisfying a strong monotonicity assumption.

\emph{Strong monotonicity.}
Finally, we assume that $\At(\cdot;\cdot)$ is strongly monotone, i.e.\ there exists a positive constant~$\Cmon$ such that 
\begin{gather}\label{eq:discrete_mono}
\begin{aligned}
\Cmon^{-1}\normt{w_\cT-v_\cT}^2 \leq \At(w_\cT;w_\cT-v_\cT)-\At(v_\cT;w_\cT-v_\cT) &&&\forall \wt,\vt\in\Vt^s.
\end{aligned}\tag{A3}
\end{gather}
In practice, strong monotonicity for DG and $C^0$-IP methods is usually attained by introducing stabilization terms and penalization terms on the jumps of the approximate solution values and gradients, and choosing the penalty parameters to be sufficiently large.
The Lipschitz continuity and strong monotonicity conditions in~\eqref{eq:discrete_cont} and~\eqref{eq:discrete_mono} are natural discrete counterparts to the Lipschitz continuity and strong monotonicity of the continuous nonlinear form $A(\cdot;\cdot)$ considered in the proof of Theorem~\ref{thm:well_posedness}.

\begin{remark}[Notion of consistency]
We call the first inequality in~\eqref{eq:discrete_consistency} \emph{discrete consistency} because it is a notion of consistency on the numerical method that is determined entirely at the discrete level.
See also~\cite{VeeserZanotti2018} for a seemingly related notion of consistency called \emph{full algebraic consistency}, which plays an important role in the analysis of abstract nonconforming methods for linear problems.
In particular, the notion of consistency employed differs from more usual notions of consistency based on inserting the exact solution~$u$ into the numerical scheme, which may be subject to additional regularity assumptions on the solution.
In practice, the discrete consistency condition~\eqref{eq:discrete_consistency} is trivially satisfied by some numerical methods, such as the one in~\cite{NeilanWu19} but is far from obvious for the original method of~\cite{SS13,SS14} owing to the additional stabilization terms.
One of our main contributions in Sections~\ref{sec:num_schemes} and~\ref{sec:proofs} below is a proof of~\eqref{eq:discrete_consistency} for the original method of~\cite{SS13,SS14} and some original variants, see in particular Theorem~\ref{thm:stab_bound} and Corollary~\ref{cor:discrete_cons}.
In all cases, our results hold without introducing any additional regularity assumptions on the exact solution.
\end{remark}

It follows from the Lipschitz continuity assumption~\eqref{eq:discrete_cont} and strong monotonicity~\eqref{eq:discrete_mono} that there exists a unique $\ut\in\Vt^s$ that solves~\eqref{eq:num_scheme}.
We now prove the main result on the \emph{a priori} error analysis of these schemes, namely a near-best approximation property akin to C\'ea's Lemma, with a constant determined solely in terms of $\dim$, $\lambda$, $\Cmon$, $C_{L_{\cT}}$ and $C_{\mathrm{cons}}$ appearing above.
Moreover, for the class of numerical methods considered below, the assumptions of our framework will be satisfied without requiring any further regularity on the exact solution.
Recall that $u\in H^2(\Om)\cap H^1_0(\Om)$ denotes the unique solution of~$F[u]=0$ and equivalently $\Fg[u]=0$ pointwise a.e.\ in $\Om$, see Theorem~\ref{thm:well_posedness}.

\begin{theorem}[Near-best approximation]\label{thm:best_approx}
Suppose that the nonlinear form $\At\colon \Vt^s\times\Vt^s\tends \R$ is linear in its second argument, and satisfies assumptions~\eqref{eq:discrete_cont},~\eqref{eq:discrete_consistency} and \eqref{eq:discrete_mono}. Let $\ut\in\Vt^s$ denote the unique solution of~\eqref{eq:num_scheme}. Then, we have the near-best approximation bound
\begin{equation}\label{eq:best_approx}
\normt{u-\ut} \leq C_{\mathrm{NB}}\inf_{\vt \in \Vt^s} \normt{u-\vt},
\end{equation}
where the constant $C_{\mathrm{NB}}$ is given by
\[
C_{\mathrm{NB}} \coloneqq 1 + \Cmon\left(C_{\mathrm{cons}}+ C_{L_{\cT}} \max\left\{1,\sqrt{2\lambda},\lambda\right\}\left(1+\sqrt{\dim+1}\right) \right).
\]
\end{theorem}

\begin{proof}
Let $\ut$ be the unique solution of~\eqref{eq:num_scheme}, and let $\vt$ be arbitrary.
Then, writing $z_{\cT}\coloneqq \vt-\ut$, we see from~\eqref{eq:num_scheme},~\eqref{eq:discrete_consistency} and \eqref{eq:discrete_mono} that
\[
\begin{split}
\Cmon^{-1}\normt{\vt-\ut}^2 = \Cmon^{-1}\normt{z_{\cT}}^2 &\leq \At(\vt;z_{\cT}) - \At(\ut;z_{\cT}) = \At(\vt;z_{\cT})
\\ & = \At(\vt;z_{\cT}) - \int_\Om \Fg[\vt]L_{\cT}z_{\cT} + \int_{\Om} \left(\Fg[\vt] - \Fg[u] \right) L_{\cT}z
\\ & \leq C_{\mathrm{cons}} \absJ{\vt-u}\normt{z_{\cT}} + \norm{\Fg[\vt]-\Fg[u]}_\Om C_{L_{\cT}}\normt{z_{\cT}}
\\ & \leq C_{\mathrm{cons}}\absJ{\vt-u}\normt{z_{\cT}} + c_{\dim,\lambda} C_{L_{\cT}} \normt{\vt-u}\normt{z_{\cT}},
\end{split}
\]
where in the second line we have added and subtracted $\int_\Om \Fg[\vt]L_{\cT}z_{\cT}$ and we have used the fact that $\Fg[u]=0$ a.e.\ in $\Om$, then in the third line we have used the discrete consistency bound~\eqref{eq:discrete_consistency} with the identity $\absJ{\vt}=\absJ{\vt-u}$, along with the Cauchy--Schwarz inequality, and in the fourth line we have $c_{\dim,\lambda}\coloneqq\max\{1,\sqrt{2\lambda},\lambda\}(1+\sqrt{\dim+1})$ which is obtained by bounding the right-hand side of~\eqref{eq:cordes_ineq2}.
We then deduce from the triangle inequality and $\absJ{\vt-u}\leq \normt{\vt-u}$ that
\begin{equation}
\normt{u-\ut}\leq \normt{u-\vt}+\normt{\vt-\ut} \leq \left[1 + \Cmon\left(C_{\mathrm{cons}}+c_{\dim,\lambda}C_{L_{\cT}} \right) \right]\normt{u-\vt},
\end{equation}
This proves~\eqref{eq:best_approx} upon taking the infimum over all $\vt\in\Vt^s$.
\end{proof}

Theorem~\ref{thm:best_approx} and the general framework introduced above can be easily extended to methods using a wide range of approximations spaces, and not only the space $\Vt^s$ considered here.

Note that the near-best approximation property given by~Theorem~\ref{thm:best_approx} is rather remarkable given the fact we consider here nonconforming methods, and is again primarily a consequence of the fact that $\Fg[u]=0$ in the strong sense. 
Whereas the quasi-optimality of \emph{conforming} Galerkin approximations of strongly monotone operator equations is classical, c.f.\ Section 25.4 of~\cite{Zeidler1990}, it is well-known that the analysis of near-best approximation properties for \emph{nonconforming} methods is rather more challenging~\cite{Gudi2010,VeeserZanotti2018}.
We refer the reader to~\cite{VeeserZanotti2018} for a detailed analysis of quasi-optimality for nonconforming methods for abstract linear problems.
Note that even in the case of linear divergence form elliptic problems, classical DG and other nonconforming methods often do not satisfy a near-best approximation property~\cite[Remark 4.9]{VeeserZanotti2018}, with the closest available results typically including additional terms on the right-hand side~\cite{Gudi2010}.

Theorem~\ref{thm:best_approx} implies that, up to associated constants, all numerical methods satisfying the assumptions of the above framework are quasi-optimal.
Provided that the constants in the assumptions~\eqref{eq:discrete_cont},~\eqref{eq:discrete_consistency} and~\eqref{eq:discrete_mono} are independent of the mesh-size, it is then easy to show optimal rates of convergence with respect to the mesh-size under additional regularity assumptions on the exact solution. Since the techniques for deriving convergence rates are rather well-known, we leave the details to the reader.

Also under the assumption that the constants in the framework above remain uniformly bounded, Theorem~\ref{thm:best_approx} then leads to convergence of the numerical solutions in the small-mesh limit without any additional regularity assumptions on $u \in H^2(\Om)\cap H^1_0(\Om)$.

\begin{corollary}[Convergence for minimal regularity solutions]\label{cor:convergence}
Let $\{\cT_k\}_{k=1}^\infty$ be a sequence of conforming simplicial meshes such that $\max_{K\in\cT_k}h_K \tends 0$ as $k\tends \infty$, and let $u_{\cT_k} \in V_{\cT_k}^s$ denote the corresponding numerical solution of~\eqref{eq:num_scheme} for each $k\in\N$.
Suppose that, for each $k\in\N$, the nonlinear form~$A_{\cT_k}(\cdot;\cdot)$ is linear in its second argument and satisfies the assumptions~\eqref{eq:discrete_cont}, \eqref{eq:discrete_consistency}, and \eqref{eq:discrete_mono} with associated constants that are uniformly bounded with respect to $k\in \N$.
Then the sequence of numerical solutions~$\{u_{\cT_k}\}_{k\in\N}$ converges to $u \in H^2(\Om)\cap H^1_0(\Om)$ the exact solution of~\eqref{eq:isaacs_pde} with
\begin{equation}\label{eq:convergence}
\lim_{k\tends \infty}\norm{u-u_{\cT_k}}_{\cT_k} =0.
\end{equation}
\end{corollary}
\begin{proof}
We prove this in two steps, first for $s=0$ and then for $s=1$.

\emph{Step~1.} Suppose momentarily that $s=0$. Under the above hypotheses that the constants in~\eqref{eq:discrete_cont}, \eqref{eq:discrete_consistency}, and \eqref{eq:discrete_mono} are uniformly bounded with respect to $k\in\N$, we infer from~\eqref{eq:best_approx} that it is enough to show that there exists a sequence of functions $v_k \in V^0_{\cT_k}$ such that $\norm{u-v_k}_{\cT_k}\tends 0$ as $k\tends \infty$.
For each $k\in\N$, we define $v_k\in V_{\cT_k}^0$ as the unique piecewise quadratic polynomial that satisfies $\int_K(u-v_k)=0$, $\int_K\nabla(u-v_k)=0$ and $\int_K\nabla^2(u-v_k) =0$ for all $K\in\cT_k$, where integration is taken component-wise for vectors and matrices. Note here that $p\geq 2$ implies that $v_k\in V_{\cT_k}^0$.
In particular, it is easily checked that an explicit formula for $v_k|_K$ is given by $v_k|_K(x) = r + \bm{d}\cdot x + \tfrac{1}{2} x^\top\bm{H} x $ for all $x\in K$, with coefficients $r\in \R$, $\bm{d}\in \R^\dim$ and $\bm{H}\in\R^{\dim\times\dim}$, with $\bm{H}=\overline{\nabla^2 u}|_K$, $\bm{d}=\overline{\nabla u - \bm{H} x}|_K$ and $r = \overline{u - \bm{d}{\cdot}x - \tfrac{1}{2}x^\top \bm{H}x}|_K $ where $\overline{w}|_K$ denotes the mean-value of a scalar-, vector- or matrix-valued function $w$ over $K$.
It then follows from repeated applications of Poincar\'e's inequality that $\int_K \hT^{2m-4}\abs{\nabla^m(u-v_k)}^2 \lesssim \int_K \abs{\nabla^2 u- \overline{\nabla^2 u}|_K}^2 $ for each $m\in\{0,1,2\}$, for all $K\in\cT_k$ and for all $k\in\N$.
Using trace inequalities to bound the jump-seminorms $\abs{v_k}_{J,\cT_k}=\abs{u-v_k}_{J,\cT_k}$ , we then see that $\norm{u-v_k}_{\cT_k}^2 \lesssim \sum_{K\in\cT_k} \int_K\abs{\nabla^2 u- \overline{\nabla^2 u}|_K}^2 $ for all $k\in\N$.
It then follows from density of the space $C^\infty_0(\Om;\R^{\dim\times\dim})$ of smooth compactly supported $\R^{\dim\times\dim}$-valued functions in $L^2(\Om;\R^{\dim\times\dim})$ that $\sum_{K\in\cT_k} \int_K\abs{\nabla^2 u- \overline{\nabla^2 u}|_K}^2 \tends 0$ as $k\tends \infty$ and hence also that $\norm{u-v_k}_{\cT_k}\tends 0$ as $k\tends \infty$.
This implies~\eqref{eq:convergence} for $s=0$.

\emph{Step~2.} For $s=1$, let $v_k \in V_{\cT_k}^0$ define the above piecewise quadratic approximation, and let $\widetilde{v}_k = E_1 v_k \in V^1_{\cT_k}$ denote its $H^1_0$-conforming enrichment, where it is recalled that $E_1$ is as in the proof of Theorem~\ref{thm:poincare}; in a slight abuse of notation, we do not indicate here the dependence of $E_1$ on $k$.
It is straightforward then to use triangle inequalities and the bound~\eqref{eq:enrichment_operator_1} to show that $\norm{u-\widetilde{v}_k}_{\cT_k}\tends 0$ as $k\tends \infty$, thus showing~\eqref{eq:convergence} also in the case $s=1$.
\end{proof}

\section{Application to a family of numerical methods}\label{sec:num_schemes}

We now consider how the abstract framework for analysis in the sections above applies to a family of numerical methods that includes as special cases the methods of~\cite{SS13,SS14,NeilanWu19} as well as some original methods which are studied further in the context of adaptive methods in~\cite{KaweckiSmears20adapt}.

\paragraph{{\bfseries Lifting operators.}} Let $q$ denote a fixed choice of polynomial degree such that $q\geq p-2$, which implies that $q\geq 0$ since $p\geq 2$. 
Let $V_{\cT,q} \coloneqq  \{w \in L^2(\Om)\colon w|_K \in \mathbb{P}_q \;\forall K \in\cT\}$ denote the space of piecewise polynomials of degree at most $q$ over $\cT$.
For each interior face $F\in\cF^I$, we define the lifting operator $r_{\cT}^F\colon L^2(F)\tends V_{\cT,q}$  by $\int_\Om r_{\cT}^F(w) \varphi = \int_F w \{\varphi\}$ for all $\varphi \in V_{\cT,q}$ and all $w\in L^2(F)$.
 Using an inverse inequality for polynomials, it is easy to see that $\norm{r^F_{\cT}(w)}_\Om\lesssim h_F^{-1/2} \norm{w}_F$ for any $w\in L^2(F)$.

For a fixed choice of a parameter $\chi\in\{0,1\}$, we define the linear operators $\Dmut \colon \Vt^s\tends L^2(\Om)$ and $r_{\cT}\colon \Vt^s\tends V_{\cT,q}$
 \begin{equation}\label{eq:lift_lapl_def}
\Dmut \vt\coloneqq \Delta \vt - \chi r_{\cT}(\jump{\nabla \vt\cdot\bn}), \quad r_{\cT}(\jump{\nabla \vt\cdot\bn})\coloneqq \sum_{F\in\cF^I}r^F_{\cT}(\jump{\nabla \vt\cdot \bn})\quad \forall \vt\in\Vt^s,
\end{equation}
In order to alleviate the notation, we do not write explicitly the dependence of $\Dmut$ on the parameter $\chi$.
If $\chi=0$ then $\Dmut \vt$ coincides with the piecewise Laplacian of $\vt$, whereas if $\chi=1$ then $\Dmut \vt$ is usually called the lifted Laplacian.
The choice $\chi=1$ is useful for proving asymptotic consistency of the numerical schemes in the context of adaptive methods, see~\cite{KaweckiSmears20adapt}. 
It is straightforward to show that 
\begin{equation}\label{eq:lifting_bound}
\begin{aligned}
\norm{ r_{\cT}(\jump{\nabla \vt\cdot\bn}) }_\Om \lesssim \absJ{\vt}, &&& \norm{\Dmut \vt}_{\Om} \lesssim \normt{\vt} &&&\forall \vt\in\Vt^s,
\end{aligned}
\end{equation}
where the constants depend only on $\dim$, $p$, $q$ and $\shapereg$, see e.g.\ \cite[Section~4.3]{DiPietroErn12}.

\paragraph{{\bfseries Stabilization.}}
Let the stabilization bilinear form $S_{\cT}\colon \Vt^s\times\Vt^s\tends \R$ be defined by
\begin{equation}\label{eq:S_def}
\begin{split}
\St (\wt,\vt)\coloneqq &\int_\Om \left[ \Dpw \wt:\Dpw \vt - \Delta \wt\Delta \vt\right]
\\ &-  \int_{\cF} \left[ \Npw_T\avg{\Npw \wt\cdot \bn} \cdot \jump{\Npw_T \vt} + \Npw_T\avg{\Npw \vt\cdot \bn} \cdot \jump{\Npw_T \wt} \right]
\\ &+ \int_{\cF^I} \left[\avg{\Delta_T \wt} \jump{\Npw \vt\cdot \bn} + \avg{\Delta_T \vt} \jump{\Npw \wt\cdot \bn} \right] 
 \quad \forall \wt,\,\vt\in\Vt^s,
\end{split}
\end{equation}
where it is recalled that $\nabla_T$ and $\Delta_T$ denote the tangential gradient and Laplacian, respectively, on mesh faces.
We now show that the stabilization form $\St(\cdot,\cdot)$ is equivalent to the stabilization terms that were used in~\cite{SS13,SS14}. In particular, let $B_{\cT,*}(\cdot,\cdot)\colon \Vt^s\times\Vt^s\tends \R$ be the bilinear form introduced in~\cite{SS13,SS14}, defined by
\begin{equation}\label{eq:Bstar_def}
\begin{split}
B_{\cT,*}(w_\cT,v_\cT)  \coloneqq & \int_\Om\left[ \nabla^2 \wt:\nabla^2 \vt+2\lambda\nabla \wt\cdot\nabla \vt+\lambda^2w_\cT v_\cT\right]
\\ &-  \int_{\cF} \left[ \Npw_T\avg{\Npw \wt\cdot \bn} \cdot \jump{\Npw_T \vt} + \Npw_T\avg{\Npw \vt\cdot \bn} \cdot \jump{\Npw_T \wt} \right]
\\ & + \int_{\cF^I} \left[\avg{\Delta_T \wt} \jump{\Npw \vt\cdot \bn} + \avg{\Delta_T \vt} \jump{\Npw \wt\cdot \bn} \right]
\\ &-\lambda\int_{\cF}\left[\avg{\Dp{w_\cT}\cdot\bn}\jump{v_\cT}+\avg{\Dp{v_\cT}\cdot\bn}\jump{w_\cT}\right]
\\ & -\lambda\int_{\cF^I}\left[\avg{w_\cT}\jump{\Dp{v_\cT}\cdot\bn}+\avg{v_\cT}\jump{\Dp{w_\cT}\cdot\bn }\right].
\end{split}
\end{equation}

The following Lemma shows that the stabilization used in~\cite{SS14} can be equivalently simplified to the stabilization form $\St(\cdot,\cdot)$ defined above in~\eqref{eq:S_def}. 

\begin{lemma}\label{lem:stab_identity}
Let the bilinear forms $\St(\cdot,\cdot)$ and $B_{\cT,*}(\cdot,\cdot)$ be defined by~\eqref{eq:S_def} and~\eqref{eq:Bstar_def}. Then, we have the identity 
\begin{equation}\label{eq:stab_identity}
\begin{aligned}
\St (\wt,\vt) = B_{\cT,*}(\wt,\vt) - \int_{\Om} L_{\lambda} \wt L_{\lambda} \vt &&& \forall \wt,\,\vt\in\Vt^s.
\end{aligned}
\end{equation}
\end{lemma}
\begin{proof}
After expanding $L_\lambda \wt L_\lambda \vt=\Delta \wt\Delta \vt - \lambda \Delta \wt \vt - \lambda \wt\Delta \vt + \lambda ^2 \wt\vt$, we see that the identity in~\eqref{eq:stab_identity} follows straightforwardly from the integration-by-parts identity
\begin{equation}\label{eq:lap_identity}
\begin{aligned}
-\int_{\Om} \Delta \wt \vt = \int \nabla \wt\cdot \nabla \vt 
 - \int_{\cF} \avg{\nabla \wt\cdot \bn }\jump{\vt} - \int_{\cF^I} \jump{\nabla \wt\cdot\bn}\avg{\vt} &&&\forall \wt,\vt\in \Vt^0,
\end{aligned}
\end{equation}
which is used twice, once as above and once with $\wt$ and $\vt$ interchanged, in order to cancel all terms involving $\lambda$ in the right-hand side of~\eqref{eq:stab_identity}.
\end{proof}

Lemma~\ref{lem:stab_identity} shows that the stabilization terms used in~\cite{SS14} for $\lambda$ possibly nonzero in fact coincides with the stabilization term used below in~\eqref{eq:At_def} that defines the nonlinear form $\At(\cdot;\cdot)$. Therefore, in practice, the method in~\cite{SS13,SS14} only requires the implementation of the terms of the stabilization form~$\St(\cdot,\cdot)$.

\paragraph{{\bfseries Penalization.}}

For two positive constant parameters $\sigma$ and $\rho$ to be chosen later, let the jump penalization bilinear form $\Jpen\colon \Vt^s\times\Vt^s\tends \R$ be defined by
\begin{equation}\label{eq:jump_pen_def}
\begin{aligned}
\Jpen(\wt,\vt)\coloneqq \int_{\cF^I} \sigma \hT^{-1} \jump{\Npw \wt}\cdot\jump{\Npw \vt} + \int_{\cF^B} \sigma \hT^{-1}\jump{\nabla_T \wt}\cdot\jump{\nabla_T \vt} + \int_{\cF} \rho \hT^{-3}\jump{\wt}\jump{\vt},
\end{aligned}
\end{equation}
for all $\wt$, $\vt\in\Vt^s$, where it is recalled that $\nabla_T$ denotes the tangential gradient on mesh faces.

\begin{remark}[Penalization of jumps of tangential gradients]
The bilinear form $\Jpen(\cdot,\cdot)$ includes terms that penalize tangential jumps of the solution on interior and boundary faces.
For fixed polynomial degrees, it is straightforward to show that the jump penalization bilinear form $\Jpen(\cdot,\cdot)$ induces a semi-norm that is equivalent to $\absJ{\cdot}$, up to constants depending on the penalty paramters $\sigma$ and $\rho$.
However, the benefit of the terms that penalize explicitly the jumps in tangential components of the gradients in the numerical scheme is that it significantly improves the dependence of the penalty parameters on the polynomial degrees, in particular $\rho$, which is essential for avoiding a degradation of the rate of convergence with respect to polynomial degrees in the context of $hp$-version methods, and it also helps to improve the conditioning of the systems, see the analysis in~\cite{SS13,SS14}.
Thus the inclusion of explicit penalization of the jumps of tangential components is advantageous in computational practice even though it is not strictly necessary for an analysis that is not explicit in the polynomial degrees.
Note however that for $C^0$-IP methods, i.e.\ when $s=1$, then the last two terms in~\eqref{eq:jump_pen_def} vanish identically.
\end{remark}

\paragraph{{\bfseries Numerical methods.}}
Recalling the operator $\Dmut$ from~\eqref{eq:lift_lapl_def}, we define the linear operator 
\begin{equation}
\begin{aligned}
\LLt \vt \coloneqq \Dmut \vt - \lambda \vt &&&\forall \vt\in\Vt^s.
\end{aligned}
\end{equation}
As above, we do not indicate explicitly the dependente of $\LLt$ on $\chi$ in order to alleviate the notation.
We now consider the following family of numerical methods: for a parameter $\theta \in [0,1]$, define the nonlinear form
\begin{equation}\label{eq:At_def}
\begin{aligned}
\At(\wt;\vt)\coloneqq \int_{\Om}\Fg[\wt] L_{\lambda,\cT} \vt + \theta \St(\wt,\vt) + \Jpen(\wt,\vt) &&& \forall \wt,\,\vt\in\Vt^s.
\end{aligned}
\end{equation}
For simplicity of notation, we do not write explicitly the dependence of $\At(\cdot;\cdot)$ on the parameters $\lambda$, $\theta$, $\chi$, $\sigma$, $\mu$ and $\rho$, nor on choice of approximation space through $s\in\{0,1\}$ and the polynomial degrees $p$ and $q$.
The discrete problem is then to find $\ut\in\Vt^s$ that solves~\eqref{eq:num_scheme}.

\begin{remark}[Relation to methods in the literature]\label{rem:relation_literature}
Choosing $s=0$, $\chi=0$ and $\theta=1/2$, we obtain the original DGFEM proposed in~\cite{SS13,SS14}, see~Lemma~\ref{lem:stab_identity} concerning the equivalence of the stabilization terms.
If we take $s=1$, and $\chi=\theta=0$, then we obtain the $C^0$-interior penalty FEM proposed in~\cite{NeilanWu19}, and further analysed in~\cite{Kawecki19c}.
Methods using $\chi=1$ are of interest in the context of adaptive methods, see~\cite{KaweckiSmears20adapt}.
Note however that the general framework of Section~\ref{sec:framework} applies to some methods not directly covered by the class of methods of this section, such as one of the two methods proposed in~\cite{Bleschmidt19}, which involves a $C^0$-IP method featuring a Hessian recovery into discontinuous piecewise polynomials for both trial and test functions.
\end{remark}

We now state the main results that show that the family of numerical methods considered above satisfy the assumptions~\eqref{eq:discrete_cont},~\eqref{eq:discrete_consistency} and~\eqref{eq:discrete_mono}  of the abstract framework for \emph{a priori error analysis}.

\paragraph{{\bfseries Lipschitz continuity.}}
Using the same techniques as in~\cite{SS13,SS14} and using Lemma~\ref{lem:cordes_ineq}, it can be shown that the nonlinear form $\At(\cdot;\cdot)$ defined in~\eqref{eq:At_def} satisfies the Lipschitz continuity bound~\eqref{eq:discrete_cont}.
In particular, Lemma~\ref{lem:cordes_ineq} improves on~\cite{SS13,SS14} by showing that the Lipschitz constant is otherwise independent of the data of the operators $L^{\ab}$.

\begin{lemma}[Lipschitz continuity]\label{lem:lipschitz}
The nonlinear form $\At(\cdot;\cdot)$ defined by~\eqref{eq:At_def} satisfies~\eqref{eq:discrete_cont} with a constant $C_{\mathrm{Lip}}$ that depends only on $\dim$, $\shapereg$, $p$, $q$, $\lambda$, $\sigma$ and $\rho$. 
\end{lemma}

\begin{proof}
Let $\wt$, $\zt$ and $\vt\in\Vt^s$ be arbitrary. Then, using~\eqref{eq:cordes_ineq2} for the nonlinear terms, using~\eqref{eq:lifting_bound} for the lifting terms, and applying inverse inequalities to the face terms in the bilinear form $\St(\cdot,\cdot)$, it is found that
\[
\begin{split}
\abs{\At(\wt;\vt)-\At(\zt;\vt)}&  \leq \int_\Om \abs{\Fg[\wt]-\Fg[\vt]}\abs{L_{\lambda,\cT} \vt }  \\\ &\quad+ \abs{\St(\wt-\zt,\vt)}+\abs{\Jpen(\wt-\zt,\vt)}
\\ &\lesssim \normt{\wt-\zt} \normt{\vt},
\end{split}
\]
with a constant depending on $\dim$, $\shapereg$, $p$, $q$, $\lambda$, $\sigma$ and $\rho$, thereby proving~\eqref{eq:discrete_cont}.
\end{proof}

\paragraph{{\bfseries Discrete consistency.}}

When $\theta=0$, it is straightforward to show that the nonlinear form $\At(\cdot,\cdot)$ defined in~\eqref{eq:At_def} satisfies the discrete consistency assumption~\eqref{eq:discrete_consistency}.
However for $\theta \neq 0$ this is far from obvious. The key for showing discrete consistency is then the following bound on the stabilization term $\St(\cdot,\cdot)$, showing that $\St(\cdot,\cdot)$ is bounded with respect to the jump seminorms of its arguments, rather than the whole norm. 

\begin{theorem}[Bound on stabilization terms]\label{thm:stab_bound}
The bilinear form~$\St$ defined in~\eqref{eq:S_def} satisfies the bound
\begin{equation}\label{eq:stab_bound}
\begin{aligned}
\abs{\St(\wt,\vt)}\lesssim \absJ{\wt}\absJ{\vt} &&&\forall \wt,\,\vt\in\Vt^s.
\end{aligned}
\end{equation}
\end{theorem}

The proof of Theorem~\ref{thm:stab_bound} is given in Section~\ref{sec:stab_bound} below.
We now show how it is used to prove~\eqref{eq:discrete_consistency}. 

\begin{corollary}[Discrete Consistency]\label{cor:discrete_cons}
The nonlinear form~\eqref{eq:At_def} satisfies~\eqref{eq:discrete_consistency} with a constant $C_{\mathrm{cons}}$ that depends only on $\dim$, $\shapereg$, $p$, $\sigma$, $\mu$, $\rho$, and $\Omega$, and a constant $C_{L_{\cT}}$ that depends only on $\dim$, $\lambda$ and, if $\chi=1$, then also on $\shapereg$, $p$ and $q$.
\end{corollary}
\begin{proof}
Choosing $L_{\cT}\vt \coloneqq \Dmut\vt - \lambda \vt$ for all $\vt\in\Vt^s$, we see that $\norm{L_{\cT}\vt}_\Om\lesssim \normt{\vt}$ for all $\vt\in\Vt^s$ with a constant $C_{L_{\cT}}$ that depends only on $\dim$, $\lambda$, and also $\shapereg$, $p$ and $q$ if $\chi=1$.
Then, for all $\wt$, $\vt \in\Vt^s$, we obtain
\begin{equation}
\begin{aligned}
\left\lvert \At(\wt,\vt)-\int_\Om\Fg[\wt]L_{\cT}\vt\right\rvert\leq \abs{\St(\wt,\vt)} + \abs{\Jpen(\wt,\vt)} \leq C_{\mathrm{cons}} \absJ{\wt}\absJ{\vt},
\end{aligned}
\end{equation}
where we have used $\theta\in[0,1]$, and we have used Theorem~\ref{thm:stab_bound} in the second inequality to bound $\abs{\St(\wt,\vt)}$.
The constant $C_{\mathrm{cons}}$ above depends only on $\dim$, $\shapereg$, $p$, on the penalty parameters $\sigma$ and $\rho$, and on $\Omega$.
This proves~\eqref{eq:discrete_consistency}.
\end{proof}

\begin{remark}
The fact that the seminorm $\absJ{\vt}$ appears on the right-hand side of~\eqref{eq:stab_bound} for the function $\vt$ in the second argument of the bilinear form $\St(\cdot,\cdot)$ is not strictly necessary for the discrete consistency property~\eqref{eq:discrete_consistency}. Indeed, the condition~\eqref{eq:discrete_consistency} allows the full norm of the second argument of the nonlinear form to appear on the right-hand side. Thus, it is possible to show that the discrete consistency assumption~\eqref{eq:discrete_consistency} also holds for a nonsymmetric variant of the stabilization term $\St(\cdot,\cdot)$.
\end{remark}

\paragraph{{\bfseries Strong monotonicity.}}

We now show below that for all choices of the parameters defining the scheme, it is possible to choose the penalty parameters sufficiently large such that~\eqref{eq:discrete_mono} is satisfied.
The analysis suggests however that the minimum necessary penalty parameters required for strong monotonicity may depend significantly on the value of the stabilization parameter $\theta$.
Indeed, Theorem~\ref{thm:strong_mono_theta} shows that if the parameter $\theta$ is in an interval centred on $1/2$, see~\eqref{eq:stab_condition} below, then~\eqref{eq:At_def} holds for a choice of penalty parameters that is independent of the geometry of the domain $\Omega$.
In Theorem~\ref{thm:strong_mono}, we show that strong monotonicity can still be achieved for general $\theta$, but with penalty parameters that possibly further depend on the geometry of $\Omega$. In the following, recall that $\CPF$ is the constant in~\eqref{eq:Poincare}.

\begin{theorem}[Strong monotonicity~I]\label{thm:strong_mono_theta}
Suppose that~$\theta$ satisfies the condition
\begin{equation}\label{eq:stab_condition}
\theta \in \left(\frac{1-\sqrt{\nu}}{2},\frac{1+\sqrt{\nu}}{2} \right),
\end{equation}
and define the positive constant $\mu>0$ by
\begin{equation}\label{eq:mu_def}
\mu\coloneqq \theta- \frac{1-\nu}{4(1-\theta)}.
\end{equation}
Then, there exists $\smin$ and $\rhomin$, depending only on $\dim$, $\shapereg$, $\lambda$, $p$, $q$, $\theta$ and $\mu$, but not on $\Omega$, such that, for all $\sigma\geq \smin$ and $\rho\geq \rhomin$, the nonlinear form $\At(\cdot;\cdot)$ satisfies~\eqref{eq:discrete_mono} with a constant $\Cmon$ depending only on $\mu$ and on $\CPF$.
\end{theorem}

\begin{proof}
Note that~\eqref{eq:stab_condition} and $\nu\leq 1$ imply that $\theta\in(0,1)$ so that $\mu$ is well-defined and real.
It is then easy to check that $\mu$ is positive if and only if $\theta$ satisfies~\eqref{eq:stab_condition}.
The proof is an extension of the approach first introduced in~\cite{SS13,SS14}.
Let $\wt$, $\vt\in\Vt^s$ be arbitrary, and let $\zt\coloneqq \wt-\vt$.
To show~\eqref{eq:discrete_mono}, we start by proving that
\begin{equation}\label{eq:strong_mono_1}
\begin{aligned}
\frac{\mu}{4}\left(\absL{\zt}^2+\absJ{\zt}^2\right) \leq \At(\wt;\zt)-\At(\vt;\zt),
\end{aligned}
\end{equation}
where we recall that the $\lambda$-weighted seminorm $\absL{\cdot}$ is defined in~\eqref{eq:lambda_seminorm}.
The Poincar\'e--Friedrichs inequality of Theorem~\ref{thm:poincare} then implies~\eqref{eq:discrete_mono}, e.g.\ with a constant $\Cmon\leq 4 \CPF^2 \mu^{-1} $.
Note that since $\nu\leq 1$, it follows from~\eqref{eq:stab_condition} that $\theta\in (0,1)$.
We then use Lemma~\ref{lem:stab_identity} to obtain
\[
\begin{split}
\At(\wt;\zt)-\At(\vt;\zt)
 = &\int_{\Om}(\Fg[\wt]-\Fg[\vt])\left(L_\lambda \zt-\chi r_\cT(\jump{\nabla\zt\cdot\bn}) \right)
 \\ & +  \theta \left( B_{\cT,*}(\zt,\zt) - \int_\Om \abs{L_\lambda \zt}^2 \right) + \Jpen(\zt,\zt).
\end{split}
\]
Adding and subtracting~$\int_\Om \abs{L_\lambda \zt}^2$ and using the bounds~\eqref{eq:cordes_ineq1},~\eqref{eq:cordes_ineq2} and~\eqref{eq:lifting_bound}, we find that
\[
\begin{split}
\At(\wt;\zt)-\At(\vt;\zt) \geq \theta B_{\cT,*}(\zt,\zt) + (1-\theta)\norm{L_\lambda \zt}^2 + \Jpen(\zt,\zt)
\\ - \sqrt{1-\nu}\absL{\zt}\norm{L_\lambda \zt}_\Om - c_{\dagger} \chi\absL{\zt}\absJ{\zt},
 \end{split}
\]
where the constant $c_{\dagger}$ depends only on $\dim$, $\shapereg$, $p$ and $q$.
Since $\theta \in (0,1)$, we may use Young's inequality $$\sqrt{1-\nu}\absL{\zt}\norm{L_\lambda\zt}_\Om \leq \frac{1-\nu}{4(1-\theta)}\absL{\zt}^2+(1-\theta)\norm{L_\lambda\zt}^2_\Om$$
to obtain
\[
\At(\wt;\zt)-\At(\vt;\zt) \geq \theta B_{\cT,*}(\zt,\zt) - \frac{1-\nu}{4(1-\theta)}\absL{\zt}^2 - c_{\dagger} \chi   \absL{\zt}\absJ{\zt} + \Jpen(\zt,\zt).
\]
It is shown in~\cite[Lemma~6]{SS14} that for any $\kappa>1$, there exists $\smin$ and $\rhomin$, depending only on $\kappa$, $\dim$, $\shapereg$, $p$ and $\lambda$, such that
\begin{equation}\label{eq:Bstar_lower_bound}
\begin{aligned}
B_{\cT,*}(\zt,\zt)+\Jpen(\zt,\zt)\geq \frac{1}{\kappa}\absL{\zt}^2 + \frac{1}{2}\Jpen(\zt,\zt) &&&\forall\zt\in \Vt^s,
\end{aligned}
\end{equation}
for all $\sigma\geq \smin$ and $\rho\geq \rhomin$.
Recalling the definition of $\mu$ in~\eqref{eq:mu_def}, we then choose, e.g., $\kappa=(1-\mu/2\theta)^{-1}$, and note that $\kappa\in(1,2)$, to get
\[
\begin{split}
\At(\wt;\zt)-\At(\vt;\zt) & \geq \left( \frac{\theta}{\kappa} -\theta + \mu\right)\absL{\zt}^2  - c_{\dagger} \absL{\zt} \absJ{\zt} + \left(1-\frac{\theta}{2}\right)\Jpen(\zt,\zt)
\\ &=\frac{\mu}{2}\absL{\zt}^2 - c_{\dagger} \absL{\zt} \absJ{\zt} + \left(1-\frac{\theta}{2}\right)\Jpen(\zt,\zt)
\\ &\geq \frac{\mu}{4}\absL{\zt}^2 + \frac{1}{2}\Jpen(\zt,\zt)-\frac{c_{\dagger}^2}{\mu}\absJ{\zt}^2,
\end{split}
\]
where in the last line we have used $1-\theta/2\geq 1/2$. It is then seen that there exists $\smin$ and $\rhomin$ sufficiently large, depending only on $\dim$, $\shapereg$, $p$, $q$, $\lambda$, $\theta$ and $\mu$, such that~\eqref{eq:strong_mono_1} and hence also \eqref{eq:discrete_mono} both hold for all $\sigma \geq \smin$ and $\rho\geq \rhomin$.
\end{proof}

\begin{remark}[Optimal value of $\theta$]
Maximizing $\mu$ with respect to $\theta$ leads to $\mu=1-\sqrt{1-\nu}$ for $\theta=1-\frac{1}{2}\sqrt{1-\nu}$, which always satisfies~\eqref{eq:stab_condition} whenever $\nu\in(0,1)$.
The constant $\Cmon$ is then comparable to the constant appearing in~\eqref{eq:strong_monotonicity}.
In the context of mixed methods, Gallistl \& S\"uli~\cite[Eq.~(2.8)]{Gallistl19} make a similar optimal choice of a parameter for stabilizing the curls of the approximations to the gradients.
We consider here more general values of $\theta$ however since the constant $\nu$ appearing in the Cordes condition might only be known approximately in practice. However, for $\nu$ small, the optimal value $\theta=1-\frac{1}{2}\sqrt{1-\nu}$ approaches $1/2$, which was the original choice made in~\cite{SS13,SS14}.
\end{remark}

When the condition~\eqref{eq:stab_condition} does not hold, e.g.~as in~\cite{NeilanWu19}, then we can still show strong monotonicity for sufficiently large penalty parameters, although the penalty parameters may then possibly depend on the geometry of $\Omega$. See Remark~\ref{rem:penalty} below.

\begin{theorem}[Strong monotonicity~II]\label{thm:strong_mono}
There exists $\smin$ and $\rhomin$, depending only on $\dim$, $\shapereg$, $\lambda$, $p$, $q$, $\nu$ and $\Omega$, such that for all $\sigma\geq \smin$ and $\rho\geq \rhomin$, and all $\theta\in[0,1]$, the nonlinear form $\At(\cdot;\cdot)$ satisfies~\eqref{eq:discrete_mono} with a constant $\Cmon$ that depends only on $\nu$ and $\CPF$.
\end{theorem} 

The proof of Theorem~\ref{thm:strong_mono} is given in Section~\ref{sec:miranda} below.

\begin{remark}[Dependencies of penalty parameters]\label{rem:penalty}
Ideally, the penalty parameters $\sigma$ and $\rho$ should be chosen as small as possible, which is important for the accuracy of the method and the conditioning of the discrete problems. 
Notice that the original method of~\cite{SS13,SS14} based on the choice $\theta=1/2$ satisfies~\eqref{eq:stab_condition} for all values of $\nu>0$, and thus the stability of the method in~\cite{SS13,SS14} is robust with respect to domain geometry.
Theorem~\ref{thm:strong_mono_theta} shows that robustness with respect to domain geometry extends to a range of choices of $\theta$ satisfying~\eqref{eq:stab_condition}.
The difficulty when $\theta$ does not satisfy~\eqref{eq:stab_condition}, e.g.\ as in~\cite{NeilanWu19}, is that the proof of strong monotonicity then relies on a discrete Miranda--Talenti inequality, where, to the best of our knowledge, all current proofs involve some reconstruction operators with constants that depend critically on the angles formed by faces at corner points and corner edges, see~Remark~\ref{rem:caveat} above. These constants then feed into $\smin$ and $\rhomin$, which leaves open the possibility that they may become very large on domains with very nearly flat edges. 
\end{remark}

\begin{remark}[Near-best approximation and convergence]
It follows from Lemma~\ref{lem:lipschitz}, Corollary~\ref{cor:discrete_cons} and Theorems~\ref{thm:strong_mono_theta},~\ref{thm:strong_mono} that the constants appearing in the abstract assumptions~\eqref{eq:discrete_cont},~\eqref{eq:discrete_consistency} and~\eqref{eq:discrete_mono} all hold with constants depending only on the quantities detailed above, and otherwise independent of the mesh-size. Therefore, Theorem~\ref{thm:best_approx} and Corollary~\ref{cor:convergence} show quasi-optimality of the approximations and convergence for minimal regularity solutions in the small mesh limit for the family of methods considered above when considering shape-regular sequences of meshes.
\end{remark}

\section{Proof of Theorems~\ref{thm:stab_bound} and~\ref{thm:strong_mono}.}\label{sec:proofs}

We now turn towards the proof of Theorems~\ref{thm:stab_bound} and~\ref{thm:strong_mono}.
Our proofs are based on more general results concerning discontinuous piecewise-polynomial \emph{vector fields}.

\subsection{Enrichment of discontinuous piecewise-polynomial vector fields}\label{sec:vector_reconst}
Consider the space $\bVt$ of piecewise-polynomial vector fields of degree at most $p-1$ defined by
\begin{equation}\label{eq:bvt_def}
\bVt \coloneqq \{\bvt\in L^2(\Om; \mathbb R^\dim);\; \bvt|_K \in \mathbb{P}_{p-1}^\dim \quad\forall K \in \cT\},
\end{equation}
where $\mathbb{P}_{p-1}^\dim$ denotes the space of $\R^\dim$-valued polynomials of total degree at most $p-1$.
Note that $\nabla \vt\in \bVt$ for any $\vt\in\Vt^s$, $s\in\{0,1\}$.
Also, the fact that $p\geq 2$ implies that $\bVt$ contains at least all piecewise affine vector-valued polynomials, and thus has a nontrivial continuous subspace.
We define the norm $\normbt{\cdot}$ and seminorm $\absbJ{\cdot}$ on $\bVt$ by
\begin{equation}
\begin{aligned}
\normbt{\bvt}^2 \coloneqq \int_{\Om} \left[ \abs{\nabla \bvt}^2+\abs{\bvt}^2 \right] + \absbJ{\bvt}^2, &&&
\absbJ{\bvt}^2\coloneqq \int_{\cF^I}\hT^{-1}\abs{\jump{\bvt}}^2+\int_{\cF^B}\hT^{-1}\abs{(\bvt)_T}^2,
\end{aligned}
\end{equation}
for all $\bvt\in\bVt$.
We also consider the space $\bH$ of $H^1$-conforming vector fields with vanishing tangential components on the boundary, i.e.\
\begin{equation}
 \bH \coloneqq \{\bm{v}\in H^1(\Om; \mathbb R^\dim);\; \bm{v}_T =0\; \text{on } \p\Omega\},
\end{equation}
where $\bm{v}_T$ denotes the tangential component of the trace of $\bm{v}$ on the boundary. 
We now construct an operator $\bE$ that maps vector fields from $\bVt$ to $\bH$-conforming vector fields, with an error controlled by the jump of the vector field over all internal faces and by the tangential component of traces over boundary faces.

\begin{theorem}\label{thm:vector_enrichment}
There exists a linear operator $\bE\colon \bVt\tends \bVt\cap \bH$ such that
\begin{equation}\label{eq:vector_enrichment}
\begin{aligned}
\normbt{\bvt-\bE\bvt}  \lesssim \absbJ{\bvt}  &&&\forall\bvt\in\bVt.
\end{aligned}
\end{equation}
The constant in~\eqref{eq:vector_enrichment} depends on $\dim$, $\shapereg$, $p$ and on $\Omega$.
\end{theorem}
\begin{proof}
The proof is an adaptation to the vectorial setting of enrichments by a standard technique of local averaging, see e.g.~\cite{KarakashianPascal03,NeilanWu19,BrennerSung2019}.
Consider the space $\bVt\cap H^1(\Om;\mathbb R^d)$ of continuous vector fields in $\bVt$, and let $\mathcal{Z}$ denote the set of points $z\in\overline{\Om}$ corresponding to the Lagrange degrees of freedom of the space $\bVt\cap H^1(\Om;\mathbb R^d)$. We remark that here the Lagrange degrees of freedom of vector fields are similar to the scalar case, with the only difference being that all degrees of freedom are point vector-values in $\R^\dim$. 
Thus, for example, if $p=2$, then $\bVt\cap H^1(\Om;\mathbb R^d)$ consists of continuous piecewise-affine vector-valued polynomials, and then $\mathcal{Z}$ consists of all mesh vertices.
Let $\mathcal{Z}$ be partitioned into the set of interior points $\mathcal{Z}^I$ and boundary points $\mathcal{Z}^B$.
For each $z\in\mathcal{Z}$, let $N(z)\coloneqq\{K\in\cT;z\in K\}$ denote the set of elements that contain $z$, where we recall that elements are by definition closed.
For each point $z\in\mathcal{Z}$, let $\cF_z \coloneqq \{F\in\cF;\;z\in F\}$ denote the set of faces containing $z$, and let $\cF^I_z\coloneqq \cF_z\cap \cF^I$ and $\cF^B_z\coloneqq \cF_z\cap \cF^B$ denote the sets of interior and boundary faces containing $z$ respectively, where we recall that faces are closed.
For boundary degrees of freedom, we distinguish two cases.
We call $z\in\mathcal{Z}^B$ flat and write $z\in\mathcal{Z}^B_{\flat}$ if and only if all of the faces in $\cF^B_z$ are coplanar. Otherwise we call $z$ sharp and write $z\in\mathcal{Z}^B_{\sharp}$. 
The operator $\bE\colon \bVt\tends\bVt\cap \bH$ is then defined in terms of its point values for each $z\in\mathcal{Z}$ by
\begin{equation}\label{eq:vector_enrichment_3}
\begin{aligned}
\bE \bvt(z) \coloneqq  
\begin{cases}
\frac{1}{\abs{N(z)}}\sum_{K'\in N(z)}\bvt|_{K'}(z)  &\text{if }z\in\mathcal{Z}^I,\\
\frac{1}{\abs{N(z)}}\sum_{K'\in N(z)}(\bvt|_{K'}(z)\cdot \bn_{\partial\Om})\bn_{\partial\Om}&\text{if }z\in\mathcal{Z}^B_\flat,\\
0 &\text{if }z\in\mathcal{Z}^B_\sharp,
\end{cases}
\end{aligned}
\end{equation}
where $\bvt\in\bVt$, where $\abs{N(z)}$ denotes the cardinality of $N(z)$, and where $\bn_{\partial\Omega}=\bn_{\partial\Omega}(z)$ denotes the unit outward normal to $\p\Omega$ at $z\in\mathcal{Z}^B_{\flat}$, which is uniquely defined when $z$ is flat.
It follows from the above definition that $\bE$ maps $\bVt$ into $\bVt\cap H^1(\Om;\R^\dim)$, and additionally it is seen that for any boundary face, $\bE \bvt$ has $\mathcal{H}^{\dim-1}$-a.e.\ vanishing tangential traces on the boundary for all $\bvt\in\bVt$, so that $\bE\colon \bVt\tends \bVt\cap \bH$.
Then, using similar arguments as in~\cite{KarakashianPascal03,NeilanWu19,HoustonSchotzauWihler07}, it is found that, for every $z\in\mathcal{Z}^I$ and every $K\in N(z)$, we have
\begin{equation}\label{eq:vector_enrichment_1}
\begin{aligned}
\abs{\bm{v}_\cT|_K(z)-\bE\bvt(z)}^2\lesssim \sum_{F\in\cF^I_z} \int_F \hT^{1-\dim}\abs{\jump{\bvt}}^2 &&&\forall \bvt\in\bVt,
\end{aligned}
\end{equation}
where the constant depends only on $\dim$, $\shapereg$ and $p$. For flat vertices $z\in\mathcal{Z}^B_{\flat}$, after splitting $\bvt(z)$ into its normal and tangential components, i.e.\ $\bvt|_K(z)=(\bvt|_K)_T+(\bvt|_K(z)\cdot\bn_{\p\Omega})\bn_{\p\Omega}$ for each $K\in N(z)$, we find that
\begin{equation}\label{eq:vector_enrichment_2}
\begin{aligned}
\abs{\bm{v}_\cT|_K(z)-\bE\bvt(z)}^2\lesssim \sum_{F\in\cF^I_z}\int_F\hT^{1-\dim}\abs{\jump{\bvt}}^2 + \sum_{F\in\cF^B_z}\int_F\hT^{1-\dim}\abs{\jump{(\bvt)_T}}^2 &&& \forall\bvt\in\bVt,
\end{aligned}
\end{equation}
where the constant depends only on $\dim$, $\shapereg$ and $p$.
Finally, if $z\in \mathcal{Z}^B_\sharp$, then there exists at least two faces in $\cF^B_z$ that are not coplanar, and thus there exists a set of unit vectors $\{\bm{t}_i\}_{i=1}^\dim$ forming a basis of $\R^\dim$, such that each $\bm{t}_i$ is a tangent vector to some face of $\cF^B_z$.
Therefore, we see that~\eqref{eq:vector_enrichment_2} also holds for $z\in \mathcal{Z}^B_{\sharp}$ but with a constant that additionally depends on the basis $\{\bm{t}_i\}_{i=1}^{\dim}$ and thus also on the geometry of $\p\Omega$, as in Remark~\ref{rem:caveat}.
The bound~\eqref{eq:vector_enrichment} is then obtained by inverse inequalities and summation of the above bounds~\eqref{eq:vector_enrichment_1} and~\eqref{eq:vector_enrichment_2}, proceeding as in~\cite{KarakashianPascal03}.
\end{proof}

\begin{remark}
The bound~\eqref{eq:vector_enrichment} can be easily improved to the sharper bound $\int_\Om\hT^{2m-2}\norm{\nabla^m(\bvt-\bE\bvt}^2 \lesssim \absbJ{\bvt}^2$ for each $m\in\{0,1\}$. However  this sharper bound is not needed in the following analysis.
Note also that the constant in~\eqref{eq:vector_enrichment} is subject to the same dependence on the geometry of the domain as the constants appearing in~\cite{NeilanWu19,BrennerSung2019}, as discussed in Remark~\ref{rem:caveat} above.
\end{remark}

\subsection{Analysis of stabilization bilinear form}\label{sec:stab_bound}

The main challenge in proving the discrete consistency bound~\eqref{eq:discrete_consistency} for the nonlinear forms $\At$ is the analysis of the stabilization bilinear form $\St(\cdot,\cdot)$ defined in~\eqref{eq:S_def}. We show here that $\St(\cdot,\cdot)$ can be seen as the restriction to piecewise gradients of a more general bilinear form on the space of vector fields $\bVt$. Let the bilinear form $\bCt\colon \bVt\times\bVt\tends \R$ be defined by
\begin{equation}\label{eq:bCt_def}
\begin{split}
\bCt(\bwt,\bvt) \coloneqq &
\int_{\Om} \left[\nabla \bwt:\nabla\bvt -(\nabla{\cdot} \bwt)(\nabla{\cdot}\bvt)-(\nabla{\times}\bwt)\cdot (\nabla{\times}\bvt)\right]
\\ &- \int_{\cF} \left[\avg{\Npw_T(\bwt{\cdot}\bn)}\cdot\jump{(\bvt)_T} + \avg{\Npw_T(\bvt{\cdot}\bn)}\cdot\jump{(\bwt)_T}\right]
\\ &+ \int_{\cF^I}\left[ \avg{\nabla_T{\cdot}(\bwt)_T}\jump{\bvt{\cdot}\bn} + \avg{\nabla_T{\cdot}(\bvt)_T}\jump{\bwt{\cdot}\bn}\right].
\end{split}
\end{equation}
where $\nabla \bvt$ denotes the density of the absolutely continuous part of $D(\bvt)$, 
where $\nabla\cdot\bvt$ denotes the trace of $\nabla \bvt$, and where, if $\dim=3$, then $(\nabla{\times}\bvt)_i\coloneqq\eps_{ijk} \nabla_{x_j}(\bvt)_{k}$ for all $i\in\{1,2,3\}$ with $\eps_{ijk}$ denoting the Levi--Civita symbol, and, if $\dim=2$, then $\nabla{\times}\bvt\coloneqq\nabla_{x_1}(\bvt)_2-\nabla_{x_2}(\bvt)_1$.
Thus, since $\bvt\in\bVt$ is piecewise smooth over $\cT$, we see that $\nabla\bvt$, $\nabla\cdot\bvt$ and $\nabla{\times}\bvt$ correspond to the piecewise gradient, divergence and curl of $\bvt$, respectively.
Using trace and inverse inequalities, it is straightforward to show that $\bCt$ is bounded on $\bVt$ in the sense that
\begin{equation}\label{eq:bCt_boundedness}
\begin{aligned}
\abs{\bCt(\bwt,\bvt)}\lesssim \normbt{\bwt}\normbt{\bvt} &&&\forall \bwt,\,\bvt\in\bVt.
\end{aligned}
\end{equation}
The bilinear form $\bCt(\cdot,\cdot)$ is related to $\St(\cdot,\cdot)$, c.f.\ \eqref{eq:S_def}, through the identity
\begin{equation}\label{eq:stab_C_relation}
\begin{aligned}
\St(\wt,\vt)=\bCt(\nabla \wt,\nabla \vt) &&&\forall \wt,\vt\in\Vt^s,
\end{aligned}
\end{equation}
which follows from the fact that the terms involving piecewise curls $\nabla{\times} \bvt$ vanish identically whenever $\bvt=\nabla \vt$ for some $\vt\in\Vt^s$.

The following Lemma can be seen as the vector-field extension of~\cite[Lemma~5]{SS13}, which was key to the consistency of the bilinear forms. In particular, it shows that $\bCt(\cdot,\cdot)$ vanishes whenever one of its arguments belongs to the subspace $\bVt\cap \bH$ of continuous piecewise-polynomial vector fields in $\bVt$ with vanishing tangential traces on $\p\Omega$.

\begin{lemma}[Consistency identity]\label{lem:consistency_identity}
For any $\bwt\in \bVt\cap \bH$ and any $\bvt\in\bvt$, we have
\begin{equation}\label{eq:consistency_identity}
\bCt(\bwt,\bvt)=\bCt(\bvt,\bwt)=0.
\end{equation}
\end{lemma}
\begin{proof}
The proof is entirely similar to~\cite[Lemma~5]{SS13}, and we include it here only for completeness. Let $\bwt \in \bVt\cap \bH$ and $\bvt\in\bVt$ be arbitrary.
Observe that since the bilinear form $\bCt(\cdot,\cdot)$ is symmetric it is enough to show that $\bCt(\bwt,\bvt)=0$.
 For each $K\in\cT$, an integration-by-parts argument implies that
\begin{multline}\label{eq:consistency_identity_1}
\int_K \left[\nabla\bwt{:}\nabla\bvt - (\nabla{\cdot}\bwt)(\nabla{\cdot}\bvt) - (\nabla{\times}\bwt)\cdot(\nabla{\times}\bvt )\right] \\ - \int_{\p K}\nabla_T  (\bwt{\cdot}\bn_{\p K}) \cdot (\bvt)_T + \int_{\p K} \nabla_T{\cdot}(\bwt)_T (\bvt{\cdot}\bn_{\p K}) =0,
\end{multline}
where $\bn_{\p K}$ denotes the unit outward normal on $\p K$.
Using the fact that tangential differential operators commute with traces, we see that $\jump{\nabla_T (\bwt{\cdot} \bn_F )}_F = \nabla_T \jump{\bwt{\cdot}\bn_F}_F=0$ for each $F\in\cF^I$, and that $\jump{\nabla_T{\cdot}(\bwt)_T}_F=\nabla_T{\cdot}\jump{(\bwt)_T}_F=0$ for each $F\in\cF$ since $\bwt$ is continuous and thus $\jump{\bwt}=0$ for all interior faces, and since $\jump{(\bwt)_T}_F=(\bwt)_T =0$ for all boundary faces $F\in\cF^B$.
Note that for each face $F\subset \p K$, we have $\bn_F = \pm \bn_{\p K}|_F$ depending on the choice of orientation of $\bn_F$, and recall that the jumps are defined by~\eqref{eq:jumpavg} in terms of this chosen orientation.
Therefore, by summing the identity~\eqref{eq:consistency_identity} and using the above identities for jumps on faces to simplify $\jump{\nabla_T  (\bwt{\cdot}\bn_{F}) \cdot (\bvt)_T}_F=\avg{\nabla_T  (\bwt{\cdot}\bn_{F})}_F\cdot\jump{(\bvt)_T}_F$ for all faces $F\in \cF$ and $\jump{\nabla_T{\cdot}(\bwt)_T (\bvt\cdot  \bn_{F})}_F=\avg{\nabla_T{\cdot}(\bwt)_T}_F\jump{\bvt\cdot  \bn_{F}}_F$ for all $F\in\cF^I$, we find that
\begin{multline*}
\int_\Om \left[\nabla\bwt{:}\nabla\bvt - (\nabla{\cdot}\bwt)(\nabla{\cdot}\bwt) - \nabla{\times}\bwt\cdot\nabla{\times}\bvt \right] \\ - \int_{\cF}\avg{\nabla_T  (\bwt{\cdot}\bn)}\cdot\jump{(\bvt)_T} + \int_{\cF^I} \avg{\nabla_T{\cdot}(\bwt)_T}\jump{\bvt{\cdot}  \bn} =0,
\end{multline*}
from which we easily obtain~\eqref{eq:consistency_identity} after noting that all remaining terms in~\eqref{eq:bCt_def} vanish since they include the jumps on normal and tangential components of $\bwt$.
\end{proof}

We now prove Theorem~\ref{thm:stab_bound}.

\emph{Proof of Theorem~\ref{thm:stab_bound}.}
We will obtain~\eqref{eq:stab_bound} as a consequence of~\eqref{eq:stab_C_relation} and the related bound
\begin{equation}\label{eq:bCt_bound}
\abs{\bCt(\bwt,\bvt)}\lesssim \abs{\bwt}_{\bm{J},\cT}\abs{\bvt}_{\bm{J},\cT}.
\end{equation}
Indeed, once~\eqref{eq:bCt_bound} is known, we deduce~\eqref{eq:stab_bound} easily from~\eqref{eq:stab_C_relation} and from the bound $\abs{\nabla \vt}_{\bm{J},\cT}\lesssim \absJ{\vt}$, which is obtained by applying the inverse inequality to the tangential component of the gradient on boundary faces, i.e.\ $\int_F\hT^{-1}\abs{\jump{(\nabla \vt)_T}}^2\lesssim \int_F \hT^{-3}\abs{\jump{\vt}}^2$ for all $F\in\cF^B$.
Therefore, it is enough to show~\eqref{eq:bCt_bound}. To do so, let $\bwt$ and $\bvt \in \bVt$ be arbitrary, and recall $\bE$ from~Theorem~\ref{thm:vector_enrichment}. Then, since $\bE\colon\bVt\tends \bVt\cap \bH$, we infer from Lemma~\ref{lem:consistency_identity} that $\bCt(\bE\bwt,\bvt)=\bCt(\bwt,\bE\bvt)=\bCt(\bE\bwt,\bE\bvt)=0$ and hence
\begin{equation}
\bCt(\bwt,\bvt)= \bCt(\bwt-\bE\bwt,\bvt-\bE\bvt) .
\end{equation}
We then apply the bounds~\eqref{eq:vector_enrichment} and~\eqref{eq:bCt_boundedness} to obtain
 $\abs{\bCt(\bwt,\bvt)} \lesssim \normbt{\bwt-\bE\bwt}\normbt{\bvt-\bE\bvt} \lesssim \abs{\bwt}_{\bm{J},\cT}\abs{\bwt}_{\bm{J},\cT}$, which gives~\eqref{eq:bCt_bound} and thus completes the proof of~\eqref{eq:stab_bound}.\qed\medskip

\subsection{Discrete Miranda--Talenti inequality and proof of Theorem~\ref{thm:strong_mono}}\label{sec:miranda}

We now turn towards the proof of Theorem~\ref{thm:strong_mono}.
The proof follows the approach based on a discrete Miranda--Talenti inequality~\cite{NeilanWu19}.
Here we remove the restriction in~\cite{NeilanWu19} that $p\leq 3$ in the case $\dim=3$, and allow instead all $p\geq 2$ for all $\dim\in\{2,3\}$.
Moreover, we show here that the discrete Miranda--Talenti inequality can be seen as a special case of a more general result for discontinuous piecewise polynomial vector fields.

\begin{theorem}\label{thm:vector_MT}
All vector fields $\bvt\in\bVt$ satisfy
\begin{equation}\label{eq:vector_MT}
\begin{aligned}
\left\lvert \left( \int_\Om \abs{\nabla \bvt}^2\right)^{\frac{1}{2}} -\left( \int_\Om\left[\abs{\nabla{\cdot} \bvt}^2+\abs{\nabla{\times} \bvt}^2\right]\right)^{\frac{1}{2}}  \right\rvert \lesssim \absbJ{\bvt}.
\end{aligned}
\end{equation}
The constant in~\eqref{eq:vector_MT} depends only on $\dim$, $\shapereg$, $p$ and $\Om$.
\end{theorem}
\begin{proof}
Let $\bvt\in\bVt$ be arbitrary, and recall the operator $\bE$ from Theorem~\ref{thm:vector_enrichment}. Since $\bE\bvt\in \bVt\cap \bH$ for all $\bvt\in\bVt$, Lemma~\ref{lem:consistency_identity} implies that $\bCt(\bE\bvt,\bE\bvt)=0$ and thus upon noting that all face integral terms in $\bCt(\bE\bvt,\bE\bvt)$ vanish identically, we get
\begin{equation}\label{eq:vector_MT_1}
\begin{aligned}
 \left(\int_{\Om}\abs{\nabla \bE\bvt}^2\right)^{\frac{1}{2}} = \left(\int_\Om\left[\abs{\nabla{\cdot} \bE\bvt}^2+\abs{\nabla{\times} \bE\bvt}^2\right]\right)^{\frac{1}{2}} .
\end{aligned}
\end{equation}
Therefore,  applying triangle and reverse triangle inequalities along with~\eqref{eq:vector_MT_1}, we deduce that
\[
\begin{split}
 &\left\lvert  \left(\int_\Om \abs{\nabla \bvt}^2\right)^{\frac{1}{2}}  - \left(\int_\Om\left[\abs{\nabla{\cdot} \bvt}^2+\abs{\nabla{\times} \bvt}^2\right]\right)^{\frac{1}{2}} \right\rvert
\\ &\qquad\leq \left\lvert \left(\int_\Om \abs{\nabla \bvt}^2\right)^{\frac{1}{2}} - \left(\int_\Om \abs{\nabla \bE\bvt}^2\right)^{\frac{1}{2}} \right\rvert
\\ &\qquad\qquad+ \left\lvert  \left(\int_\Om\left[\abs{\nabla{\cdot} \bvt}^2+\abs{\nabla{\times} \bvt}^2\right]\right)^{\frac{1}{2}} - \left(\int_\Om\left[\abs{\nabla{\cdot} \bE\bvt}^2+\abs{\nabla{\times} \bE\bvt}^2\right]\right)^{\frac{1}{2}}  \right\rvert 
\\ & \qquad\qquad\qquad \lesssim \norm{\nabla(\bvt-\bE\bvt)}_\Om \lesssim \absbJ{\bvt},
\end{split}
\]
where we have applied~Theorem~\ref{thm:vector_enrichment} in the last line, thereby proving~\eqref{eq:vector_MT}.
\end{proof}

Note that the analysis above does not use anywhere the fact that the domain is convex, and thus Theorem~\ref{thm:vector_MT} is also valid for sufficiently regular polytopal nonconvex domains.

We now see that the discrete Miranda--Talenti inequality is a direct consequence of Theorem~\ref{thm:vector_MT} by using the fact that $\nabla{\times} \bvt =0 $ whenever $\bvt = \nabla \vt$ for some $\vt\in\Vt^s$.

\begin{corollary}[Discrete Miranda--Talenti inequality]\label{cor:MT}
There exists a constant $C_{\mathrm{MT}}$ depending on $\dim$, $\shapereg$, $p$ and $\Om$, such that
\begin{equation}\label{eq:MT}
\begin{aligned}
\left\lvert\left(\int_\Om \abs{\nabla^2 \vt}^2\right)^{\frac{1}{2}} - \left(\int_\Om \abs{\Delta \vt}^2\right)^{\frac{1}{2}} \right\rvert \leq  C_{\mathrm{MT}}\absJ{\vt} &&&\forall \vt\in\Vt^s.
\end{aligned}
\end{equation}
Furthermore, for every $\delta>0$, there exists a constant $C_\delta$, depending on $\delta$, $\dim$, $p$, $\shapereg$, $\lambda$, and $\Om$, such that
\begin{equation}\label{eq:LL_lower_bound}
\begin{aligned}
(1-\delta)\absL{\vt}^2 \leq \norm{L_\lambda \vt}^2_\Om + C_\delta \absJ{\vt}^2 &&&\forall \vt\in\Vt^s.
\end{aligned}
\end{equation}
\end{corollary}
\begin{proof}
The proof of~\eqref{eq:MT} is immediate from~Theorem~\ref{thm:vector_MT}, so it remains only to prove~\eqref{eq:LL_lower_bound}.
Let $\delta>0$ be given, and let $\vt\in\Vt^s$ be arbitrary. For $\epsilon>0$ to be chosen below, we infer from Corollary~\ref{cor:MT} and Young's inequality $2xy\leq \eps x^2 + \eps^{-1}y^2$ for all positive numbers $x$, $y$, that
\begin{equation}\label{eq:LL_lower_bound_1}
\begin{aligned}
 (1+\eps)^{-1}\int_{\Om} \abs{\nabla^2 \vt}^2 \leq \int_{\Om } \abs{\Delta \vt}^2 + (1+\eps)^{-1}(1+\eps^{-1})C_{\mathrm{MT}}^2\absJ{\vt}^2&&&\forall \vt\in\Vt^s,
 \end{aligned}
\end{equation}
Moreover, using inverse inequalities and the integration by parts identity~\eqref{eq:lap_identity}, we find that{}
\begin{equation}\label{eq:LL_lower_bound_2}
\int_\Om \left[ 2\lambda \abs{\nabla\vt}^2 + \lambda^2 \abs{v}^2\right] \leq \int_\Om \left[ - 2\lambda \vt \Delta \vt + \lambda^2 \abs{\vt}^2\right] + C_{3} \absJ{\vt}\left(\int_\Om \left[ 2\lambda \abs{\nabla\vt}^2 + \lambda^2 \abs{v}^2\right]\right)^{\frac{1}{2}},
\end{equation}
where $C_3$ is a constant depending only on $\dim$, $\shapereg$, $p$ and $\lambda$.
After a further application of Young's inequality to the last term on the right-hand side of~\eqref{eq:LL_lower_bound_2}, we combine the above inequalities with~\eqref{eq:LL_lower_bound_1} and find that
\[
\begin{split}
(1+\eps)^{-1}\absL{\vt}^2 & \leq \int_{\Om} \left[\abs{\Delta \vt}^2-2\lambda \vt \Delta \vt + \lambda^2 \abs{\vt}^2\right] + C_4 \absJ{\vt}^2
 = \int_\Om \abs{\LL \vt}^2 + C_4 \absJ{\vt}^2,
\end{split}
\]
with a constant $C_4$ depending on $\eps$, $C_3$ and $C_{\mathrm{MT}}$ above. Thus, after choosing $\eps$ such that $(1+\eps)^{-1}=(1-\delta)$ for the given $\delta$, we obtain~\eqref{eq:LL_lower_bound}.
\end{proof}

We now give the proof of Theorem~\ref{thm:strong_mono}.
\medskip 

\emph{Proof of Theorem~\ref{thm:strong_mono}.}
Let $\wt$, $\vt\in\Vt^s$ be arbitrary, and let $\zt\coloneqq \wt-\vt$. Then, adding and subtracting $\norm{\LL \zt}_\Om^2$ we get
\begin{multline*}
\At(\wt;\zt)-\At(\vt;\zt) = \norm{\LL \zt}^2_\Om +\int_\Om (\Fg[\wt]-\Fg[\vt]-\LL\zt)\LL\zt 
\\ - \int_\Om (\Fg[\wt]-\Fg[\vt])\chi r_{\cT}(\jump{\nabla \zt\cdot\bn }) + \theta \St(\zt,\zt) + \Jpen(\zt,\zt)
\\ \geq \norm{\LL \zt}^2_\Om - \sqrt{1-\nu}\absL{\zt}\norm{\LL\zt}_\Om  - \chi c_{\dagger} \absJ{\zt}\absL{\zt} - \theta C_5 \absJ{\zt}^2 + \Jpen(\zt,\zt),
\end{multline*}
where we have used~\eqref{eq:cordes_ineq1}, \eqref{eq:cordes_ineq2} and Theorem~\ref{thm:stab_bound}, with $C_5$ the constant from~\eqref{eq:stab_bound}, and $c_{\dagger}$ a constant depending only on $\dim$, $\shapereg$, $p$ and $q$.
Using Young's inequality and Corollary~\ref{cor:MT} with, for instance, $\delta = \nu/4$, we then eventually find that
\[
\begin{split}
\At(\wt;\zt)-&\At(\vt;\zt)\\
&\geq \frac{1}{2}\norm{\LL \zt}^2 - \frac{(1-\nu)}{2}\absL{\zt}^2 - \chi c_{\dagger} \absJ{\zt}\absL{\zt} - \theta C_5 \absJ{\zt}^2 + \Jpen(\zt,\zt)
\\ & \geq \frac{\nu - \delta}{2} \absL{\zt}^2 - \chi c_{\dagger} \absJ{\zt}\absL{\zt} - (\theta C_5+C_{\delta}/2) \absJ{\zt}^2 + \Jpen(\zt,\zt),
\\ & = \frac{\nu}{4}\absL{\zt}^2 - \chi c_{\dagger} \absJ{\zt}\absL{\zt} - \left(\theta C_5+C_{\nu/4}/2\right) \absJ{\zt}^2 + \Jpen(\zt,\zt),
\end{split}
\]
where, after using $\chi\in\{0,1\}$ and $\theta\in[0,1]$, we see that that there exists $\smin$ and $\rhomin$ depending only on $\dim$, $\shapereg$, $p$, $q$, $\lambda$, $\nu$ and $\Omega$ such that~$\At(\wt;\zt)-\At(\vt;\zt)\geq \frac{\nu}{8}(\absL{\zt}^2+\absJ{\zt}^2)$, from which~\eqref{eq:discrete_mono} follows upon using~Theorem~\ref{thm:poincare}. In particular, we may then take $\Cmon$ to depend only on $\nu$ and $\CPF$. \qed

\section{Numerical Experiment}\label{sec:numexp}
In this section, we consider a numerical experiment for a fully nonlinear Isaacs equation posed on the irregular pentagonal domain $\Omega$ presented in Figure~\ref{fig_domain2}. Note that the domain $\Omega$ is characterized by a large interior angle $\pi-\phi$ at the origin, where $\phi\in[0,\pi/4]$ is a small parameter specified below. This choice of domain is motivated by Remark~\ref{rem:caveat}.
We consider the Isaacs equation
$$\inf_{\alpha\in \sA}\sup_{\beta\in\sB}[a^{\alpha\beta}:\nabla^2 u - f^{\alpha\beta}]=0 \quad\text{in }\Omega,$$
along with the homogeneous Dirichlet boundary condition $u=0$ on $\p\Om$, where 
\begin{equation*}
\begin{aligned}
a^{\alpha\beta}\coloneqq \beta \begin{bmatrix} \frac{1}{\sqrt{2}}(\cos\alpha+\sin\alpha) & 0\\ 0 &  \frac{1}{\sqrt{2}}(\cos\alpha-\sin\alpha)\end{bmatrix} \beta^{\top},
\end{aligned}
\end{equation*}
for all $ \alpha \in \sA\coloneqq [0,\alpha_{\mathrm{max}}]$, with $\alpha_{\mathrm{max}}\in \R_{\geq 0}$ chosen below, and for all $\beta \in \sB\coloneqq \mathrm{SO}(2)$ the special orthogonal group of matrices in $\R^{2\times 2}$.
The diffusion coefficients then satisfy the Cordes condition~\eqref{eq:Cordes2} with $\nu=\cos(2\alpha_{\mathrm{max}})\in(0,1]$, provided that $\alpha_{\mathrm{max}}<\pi/4$. In our experiment, we set $\alpha_{\mathrm{max}} \coloneqq 9\pi/40$ to be close to $\pi/4$.
The rotation matrices $\beta \in \mathrm{SO}(2)$ then have the effect of allowing the diffusion coefficients to become strongly anisotropic, and prevent the possibility of aligning the mesh with the principal directions of diffusion. Moreover, a closer analysis shows that the control parameter $\alpha$ is of \emph{bang-bang} type, leading to jump discontinuities in the optimal control.
In order to test the numerical methods in the regime of low-regularity solutions, we choose an exact solution exhibiting a singularity induced by the corner. In particular, the source term $f^{\alpha\beta}$ is chosen so that the exact solution given in polar coordinates $(r,\rho)$ is given by
\[
u(r,\rho) = -r^{\frac{\pi}{\pi-\phi}}\sin\left(\frac{\pi}{\pi-\phi}\rho\right)\eta_{1/2}(r),
\]
where $\eta_{1/2}(r):=\chi_{\{r<1/2\}}e^{1/(4r^2-1)}$ is a smooth cut-off function that is included to enforce the homogeneous Dirichlet boundary condition, with $\chi_{r<1/2}$ the indicator function for the disc of radius $1/2$ around the origin.
Note that the regularity of the solution decreases as $\phi$ becomes small.
For the computations presented below, we choose $\phi=\pi/10$, and note that $u\in H^{s}(\Omega)$ only for $s<2+1/9$, which falls outside the scope of the \emph{a priori} error analysis of some earlier works.
In particular, uniform mesh refinements would lead to low rates of convergence, so we turn to adaptive methods.

\begin{figure}
\begin{center}
\includegraphics{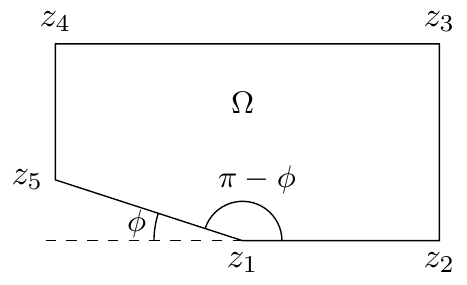}
\caption{Experiment of Section~\ref{sec:numexp}: pentagonal domain $\Omega$, with vertices $z_1=(0,0)$, $z_2=(1,0)$, $z_3=(1,1)$, $z_4=(\cos(\pi-\phi),1))$, and $z_5=(\cos(\pi-\phi),\sin(\pi-\phi))$.}
\label{fig_domain2}
\end{center}
\end{figure}

\begin{figure}
\begin{center}
\includegraphics[width=4cm]{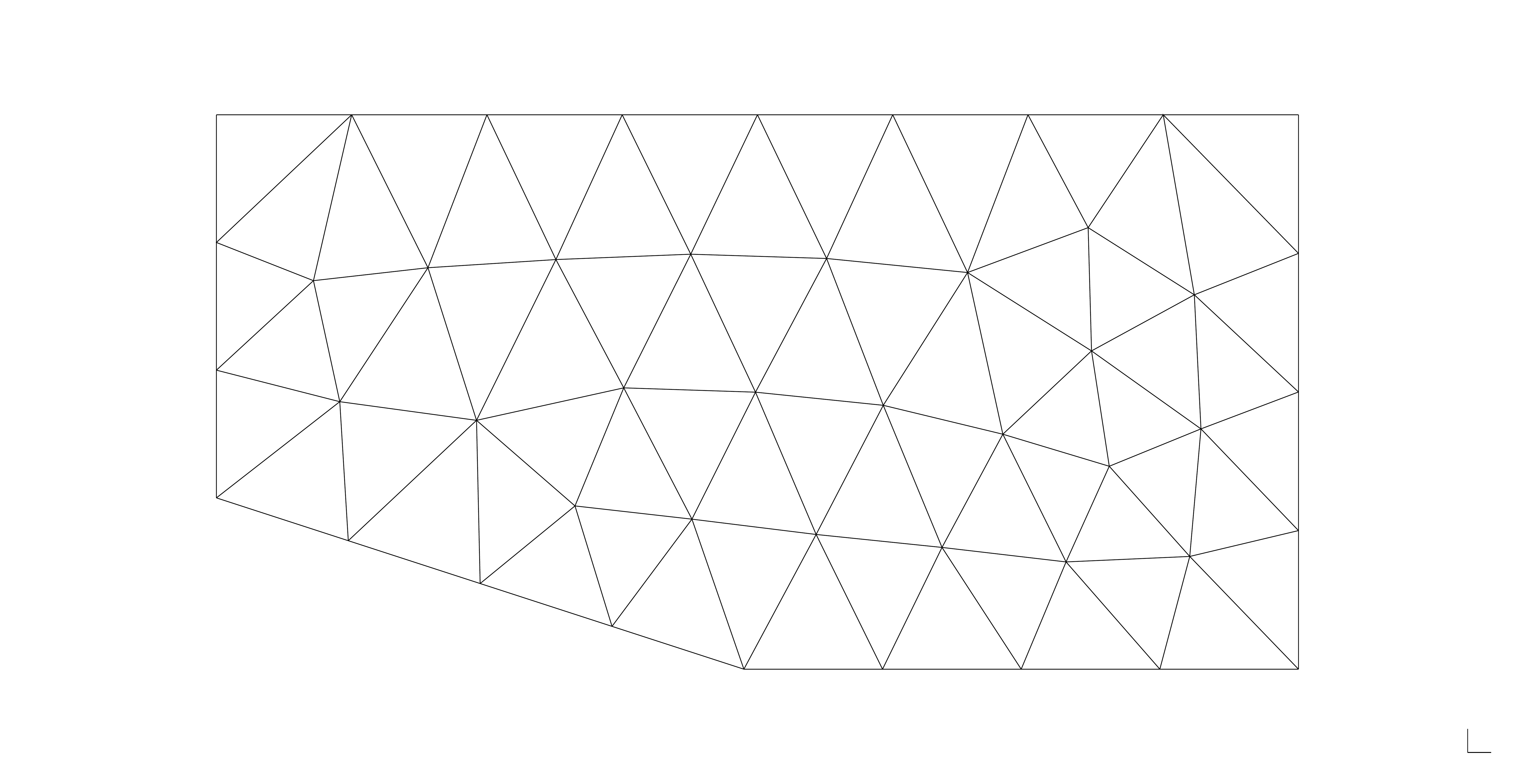}
\includegraphics[width=4cm]{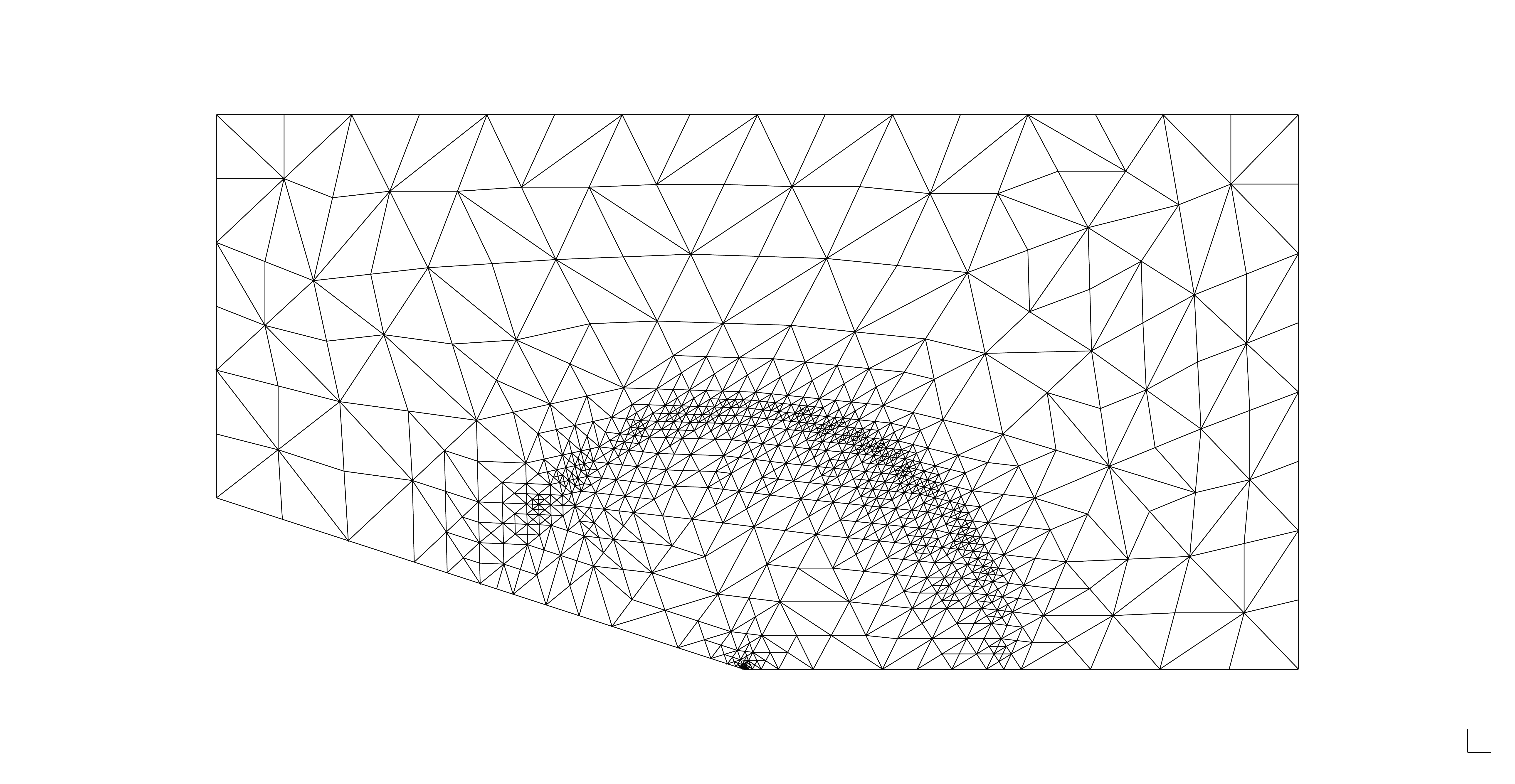}
\caption{Experiment of Section~\ref{sec:numexp}: initial mesh used for the adaptive computations (left) and sample mesh obtained after 14 steps of the adaptive method (right) using the method~\eqref{eq:At_def} with $s=1$, $p=3$, $\theta = 1/2$.}
\label{fig:meshes}
\end{center}
\end{figure}

\begin{figure}
\begin{center}
\begin{tabular}{c c}
\includegraphics{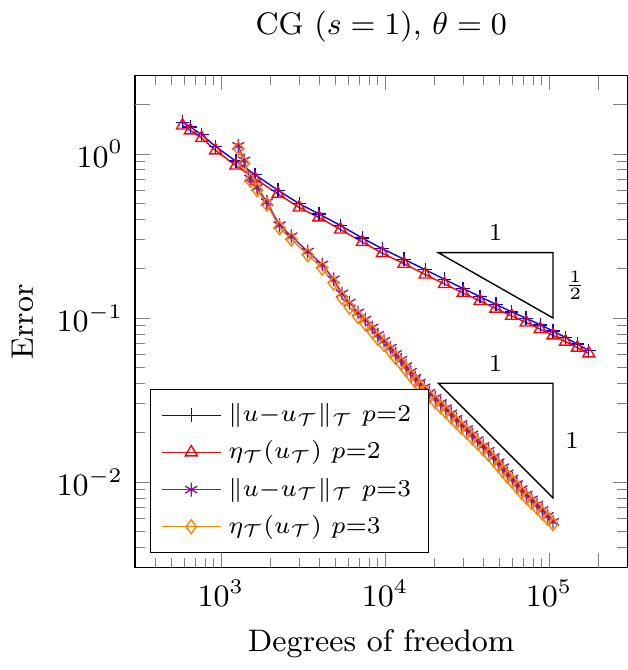} & \includegraphics{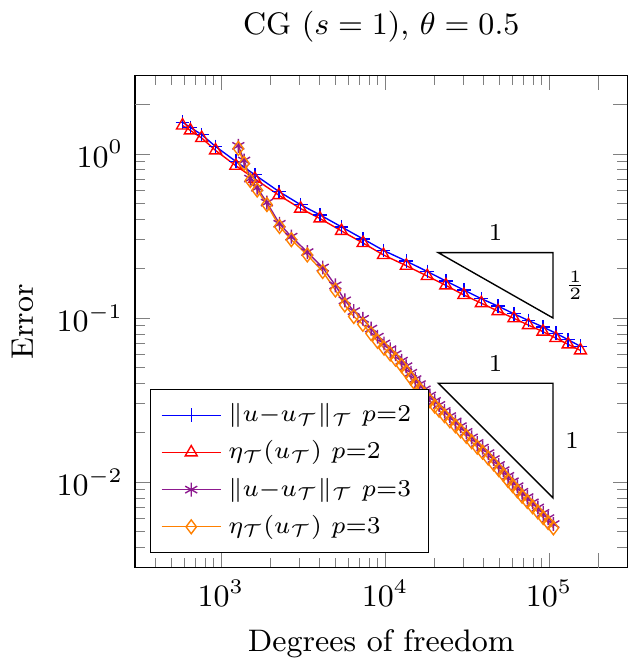} \\
\includegraphics{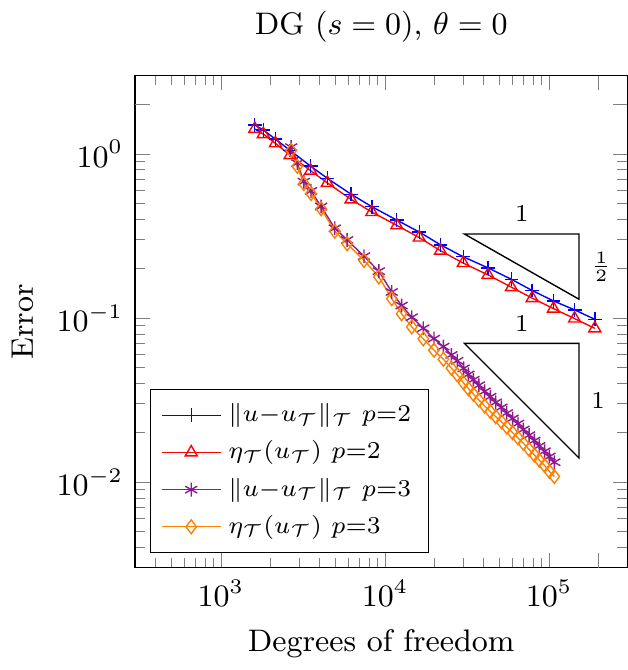} & \includegraphics{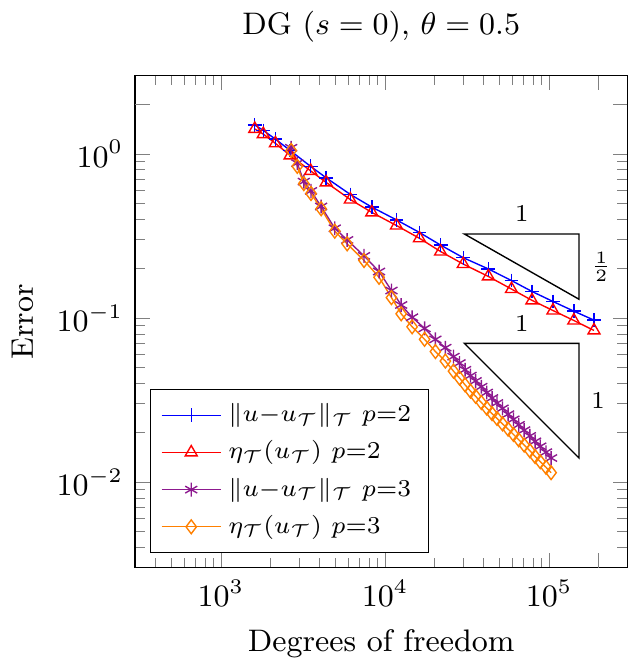}
\end{tabular}
\caption{Experiment of Section~\ref{sec:numexp}: convergence plots for a range of DG and $C^0$-IP methods of the form of \eqref{eq:At_def} on adaptively refined meshes. The convergence rates are optimal with respect to the number of degrees of freedom.}
\label{fig:4plots}
\end{center}
\end{figure}

In order to test the usefulness of the \emph{a posteriori} error estimators of Section~\ref{sec:abtract_apost}, we apply several of the methods of Section~\ref{sec:num_schemes} using adaptive mesh refinements guided by the residual error estimators~\eqref{eq:error_estimators}.
In particular, we apply a bulk-chasing (D\"{o}rfler) marking scheme with bulk-chasing parameter $1/4$.
See~\cite{KaweckiSmears20adapt} for the analysis of convergence of adaptive methods for these problems.
The coarse initial mesh used for the computations and a sample adaptively refined mesh obtained from the computations are detailed in Figure~\ref{fig:meshes}.
Our implementation is based on the software package NGSolve~\cite{NGSolve}.
Due to the nonconvexity of the Isaacs operator, the discrete nonlinear problems are solved using a Howard-type algorithm similar to~\cite[Algorithm Ho-4]{Bokanowski2009}.
This algorithm involves the solution of an outer sequence of discrete HJB problems that are each solved inexactly via inner iterations of a semismooth Newton method \cite[Section~8]{SS14}.
In our computations, we observed superlinear convergence of this algorithm with respect to both the outer and inner iterations, so that the total cost of solving the discrete Isaacs problem is comparable to the cost of solving a small number of discrete HJB equations.

Figure~\ref{fig:4plots} presents the computed errors and global error estimator values, c.f.\ \eqref{eq:error_estimators}, for a range of methods with varying parameters in the definition~\eqref{eq:At_def}. In each case, the choice $\chi=0$ is fixed, and we vary the parameters $\theta\in\{0,1/2\}$, $s\in\{0,1\}$, which corresponds to DG and $C^0$-IP methods; we also consider polynomial degrees $p\in\{2,3\}$.
In particular the case $s=0$ and $\theta=1/2$ leads to the method of~\cite{SS13,SS14} whereas $s=1$ and $\theta=0$ leads to the method of~\cite{NeilanWu19}, see Remark~\ref{rem:relation_literature}.
It is found that the adaptive algorithm leads to the optimal rates of convergence with respect to the number of degrees of freedom; indeed, for all of the methods, we obtain convergence rates of optimal order $N^{-1/2}$ for $p=2$ and of optimal order $N^{-1}$ for $p=3$, where $N$ denotes the number of degrees of freedom.
Figure~\ref{fig:4plots} further shows the efficiency of the estimators across all of the computations, with efficiency indices close to the ideal value of $1$.
It is also seen that the accuracy of the method is similar for the different values of $\theta\in\{0,1/2\}$, as may be expected from the quasi-optimality of all these methods.
We also note that we did not observe significant qualitative differences when varying the angle $\phi$ in further computations.

\bibliographystyle{spmpsci}
\bibliography{unifiedbib}

\end{document}